\documentclass[11pt,a4paper]{article}
\usepackage[english]{babel}
\usepackage{amsthm}
\usepackage{amsmath}
\usepackage{amsfonts}
\usepackage{amssymb}
\usepackage{makeidx}
\usepackage{graphicx}
\usepackage[normalem]{ulem}
\usepackage{color}
\usepackage[colorlinks=true,linkcolor=blue,citecolor=blue,urlcolor=black]{hyperref}
\usepackage[left=3.2cm,right=3.4cm,top=3cm,bottom=2.5cm]{geometry}
\usepackage{fancyhdr}
\providecommand{\ams}[1]{\textit{AMS 2010 Subject Classification.} #1}
\providecommand{\keywords}[1]{\textit{Keywords.} #1}

\def\wC{\widehat{C}}

\def\grad{\nabla}
\def\Hess{\nabla^2}
\def\P{\mathbb{P}}
\def\E{\mathbb{E}}
\newcommand{\smc}{\scshape}
\def\real{{\mathbb R}}

\def\hM{{\widehat M}}
\def\hg{{\widehat g}}
\def\Var{\text{Var}}
\def\bi{\mathbf i}
\def\Xi{\X_\bi}

\def\tM{\widetilde{M}}
\def\wM{\widehat{M}}

\def\wN{\widehat{N}}

\def\smallhalf{\mbox{ $\frac{1}{2}$}}

\def\C{{\mathbb C}}
\def\wC{\widehat{\mathbb C}}
\def\tzeta{{\widetilde{\zeta}}}
\def\Btwo{{B^{(2)}}}

\newcommand{\EXCLUDE}[1]{}

\newcommand{\remove}[1]{}
\def\P{\mathbb{P}}
\def\E{\mathbb{E}}
\newcommand{\beq}{\begin{eqnarray}}
\newcommand{\eeq}{\end{eqnarray}}
\newcommand{\beqq}{\begin{eqnarray*}}
	\newcommand{\eeqq}{\end{eqnarray*}}%\pagestyle{plain}
\def\:{:\,}

\newtheorem{theorem}{Theorem}[section]

\newtheorem{lemma}[theorem]{Lemma}

\newtheorem{assumption}[theorem]{\bf Assumption}

\def\C{{\cal C}}

\newcommand{\ep}{\varepsilon}

\def\definedas{\stackrel{\Delta}{=}}

\def\definedas{\stackrel{\Delta}{=}}

\def\convas{\stackrel{a.s. }{\to}}

\def\1{\mathbf{1}}

\def\X{\bold{X}}

\def\smallhalf{\mbox{$\frac{1}{2}$}}

\begin{document}

	\title{Convergence of the reach for a sequence of   Gaussian-embedded manifolds	}

	%\begin{aug}
	% indicate corresponding author with \corref{}

	% \author{\fnms{John} \snm{Smith}\corref{}\ead[label=e1]{smith@foo.com}\thanksref{t1}}
	% \thankstext{t1}{Thanks to somebody}
	% \address{line 1\\ line 2\\ printead{e1}}
	% \affiliation{Some University}
	
	\author{	
	Robert J.\ Adler\footnote{Electrical Engineering, Technion, Haifa, Israel. Research supported in part by URSAT, ERC Advanced Grant 320422 and SATA II, AFOSR,  FA9550-15-1-0032.},
	Sunder Ram  Krishnan\footnote{Electrical Engineering, Technion, Haifa, Israel.   Research supported in part by URSAT, ERC Advanced Grant 320422 .},\\
	        Jonathan E.\ Taylor\footnote{Statistics,  Stanford University,  Stanford, CA.  Research supported in part by SATA,   AFOSR, FA9550-11-1-0216.},
	        and   Shmuel Weinberger\footnote{Mathematics, University of Chicago, Chicago, IL.     Research supported in part by SATA,   AFOSR, FA9550-11-1-0216.}.}
	        
	\maketitle
	
	\begin{abstract}
		{\footnotesize
			Motivated by questions of manifold learning, we study a sequence of random manifolds, generated by embedding a fixed, compact manifold $M$ into Euclidean spheres of increasing dimension via a sequence of Gaussian mappings. One of the fundamental smoothness parameters of manifold learning theorems is  the reach, or critical radius, of $M$. Roughly speaking, the reach is a measure  of a manifold's departure from convexity,  which  incorporates both local  curvature and global topology.

This paper develops  limit theory for the reach of a family of random, Gaussian-embedded, manifolds, establishing both almost sure convergence for the global reach, and a  fluctuation theory for both it and its local version. The global reach converges to a constant well known  both in the reproducing kernel Hilbert space theory of Gaussian processes, as well as in their extremal theory.
}
		
	\end{abstract}
	
	\ams{ Primary: 60G15, 57N35;  Secondary: 60D05, 60G60.}
	
	\keywords{Gaussian process, manifold, random embedding, critical radius, reach, curvature, asymptotics, fluctuation theory.}

\section{Introduction}

This paper has two themes to it. One lies in the general area of the geometry of Gaussian processes, or random fields, over general spaces, and is about random embeddings. The second is more topological, and can be seen as putting probability measures on spaces of manifolds, and then studying the behavior of their reach. Both are motivated from recent results in manifold learning.

\subsection{Gaussian embeddings}
 We start with parameter spaces which will always be  $m$-dimensional, compact, smooth manifolds, without boundary, and which will be denoted by $M$.
On $M$, we define a centered, unit variance, smooth, Gaussian process $f\:M\to\real$,  the distribution of which is characterized by its covariance function
  $\C\:M\times M\to\real$. Taking  $k\geq 1$, we also define a $\real^k$-valued process
\beq
\label{embedding1}
f^k(x)\ = \ \left(f_1(x),f_2(x),\cdots,f_k(x)\right),
\eeq
made up of the first $k$ processes in  an infinite sequence of i.i.d.\ copies of $f$.
It is not hard to check  (under the mild side requirements that will be made formal later) that \eqref{embedding1} defines, with probability one,  an   embedding (i.e.\ an injective homeomorphism) $f^k(M)$ of $M$ into $\real^k$  for all 
$k\geq 2m+1$, akin to what one would expect from the Whitney embedding theorem. We call this a Gaussian embedding of $M$.

It is easy to check that the diameter of $f^k(M)$ is $O(\sqrt{k})$. Thus, to keep the embedding under control, we need to normalise it either by $\sqrt{k}$, or self-normalise by defining 
\beq
h^k(x)\ \definedas \ \frac{f^k(x)}{\|f^k(x)\|},\qquad x\in M,
\label{embedding}
\eeq
where $\|\cdot\|$ is the standard Euclidean norm, and consider the embedding $h^k(M)$, which now lies in the unit sphere $S^{k-1}$ in $\real^k$. For reasons of notational convenience, this is the embedding that we shall consider in the current paper, although we could just as well have adopted a $\sqrt{k}$ normalisation without any qualitative changes in our results, although some of the details would be different. We call $h^k(M)$ a {\it self-normalised, Gaussian, embedding} of $M$.

However, although all of $M$, $f$ and the ambient spheres are smooth, it is not so clear  how smooth these embeddings are going to be as  
 $k\to\infty$. On the one hand, the self-normalisation in \eqref{embedding} ensures that $h^k(M)$ lies in a fixed radius sphere. 
On the other hand, high-dimensional spheres are strange objects, with surface areas tending to zero as the dimension grows. Thus,    given the increasing independence added into the mapping with each new $f$ component,  it is not at all a priori clear whether the embeddings eventually become rough, and perhaps fractal, or whether  there is some sort of strong law behavior that leads to deterministic behavior in the limit. If the latter case is correct (which it is) then an associated fluctuation theory is called for.

The main results of this paper resolve these issues, at least  in the framework of the  reach
 of the self-normalised Gaussian embeddings $h^k(M)$, as $k\to\infty$.
 
 \subsection{Reach}
The modern notion of  reach seems to have appeared first in the classic paper \cite{FedererCurvature}  of Federer, in which he introduced the notion of sets with positive reach and their associated curvatures and curvature measures. In doing so, Federer was able to include, in a single framework, Steiner's tube formula for convex sets and Weyl's tube formula for $C^2$ smooth submanifolds of $\real^n$.
The importance of this framework extended, however, far beyond tube formulae, as it became clear that much of the theory surrounding convex sets could be extended to sets that were, in some sense, locally convex, and that  the reach of a set was precisely the way to quantify this property.

To be just a little more precise -- a formal definition will be given below  in Section \ref{critraddef} -- 
we start with a smooth manifold $N$ embedded in an ambient manifold $\wN$. Then
the local reach at a point $x\in N$ is the furthest distance one can travel, along any vector based at $x$ but normal to $N$ in $\wN$, without meeting a similar vector originating at another point in $N$. The (global) reach of $N$ is then the infimum of all local reaches. As such it is related to local properties of $N$ through its second fundamental form, but also to global structure, since points on $N$ that are far apart in a geodesic sense might be quite close in the metric of the ambient space $\wN$. The reach of a manifold is also known as its `critical radius' for a good geometrical reason described below, and we shall use both terms interchangeably. (See the paragraph following \eqref{globc}.)

We shall give precise definitions in the following section,  noting for now   that beyond its importance in tube formulae and other classical areas of Differential Geometry and Topology, the notion of positive reach has recently begun to play an important role in the literature of Topological Data Analysis (TDA) in general, and manifold learning via algebraic techniques in particular. We shall discuss this briefly at the end of  Section \ref{sec:importance}.

\subsection{Main results and structure of the paper}
With the terminology we have so far alluded to (but in most cases have yet to define rigorously)  let  $\theta (N,x)$ denote the local reach of a manifold $N$ at the point $x\in N$, while
\beqq
\tau \ \equiv \tau(N) \ \definedas \ \inf_{x\in N} \theta(N,x),
\eeqq
denotes the global reach of $N$. We, however, are interested in the reach of $h^k(M)$, and the  main result of this paper is Theorem \ref{maintheorem}, which states that there is a deterministic function $\sigma^2_c(f,x)$, $x\in M$, such that, with probability one, and uniformly in $x\in M$, 
\beq
\label{intro:localconv}
\cot^2 \left(\theta\left(h^k(M),h^k(x)\right)\right) \ \to\     \sigma_c^2(f,x),
\eeq
as $k\to\infty$. An immediate consequence of this is the existence of a constant, denoted by $\sigma^2_c(f)$, such that the sequence of global reaches satisfies
\beq
\label{intro:globalconv}
\cot^2 \left(\tau\left(h^k(M)\right)\right) \ \convas\     \sigma^2_c(f) \ \definedas \ \sup_{x\in M} \sigma_c^2(f,x).
\eeq
While the notation regarding the various versions of $\sigma^2$ is a little heavy, it is time-honored, since the constant $\sigma^2_c(f)$ has appeared previously in the extremal theory of Gaussian processes. In fact, one of the most interesting aspects of the convergence in \eqref{intro:globalconv} is the, a priori surprising, fact that  $\sigma^2_c(f)$ is the limit.  This constant had arisen earlier in a completely different context in \cite{Adler,Takemura}. That context, described briefly in Section \ref{sigmac2}, related to rigorously proving the so called `Euler characteristic heuristic', which approximates a wide class of Gaussian extremal probabilities via the expected Euler characteristic of their excursion sets. The role of the constant there is in quantifying the super-exponentially small error rate involved in the approximation. We shall discuss the importance of this constant in more detail in Section \ref{sec:GPs}.

Given the convergence in \eqref{intro:localconv}, it is natural to ask if an associated fluctuation result also holds. Indeed, this is the case, and Theorem \ref{maintheorem} also gives us that 
\beq
\label{intro:weakconv}
\sqrt{k}\left(\cot^2 \left(\theta\left(h^k(M),h^k(\cdot)\right)\right) \ - \  \sigma_c^2(f,\cdot)\right)
\eeq
converges, in distribution, as $k\to \infty$, to a limit which can be bounded by the supremum of a certain  Gaussian process, the precise distribution of which is given much later in Theorem  \ref{thm:mainweak}.

The remainder of the paper is organised as follows: In the following section we have collected some general results about positive reach that were a large part of the  motivation for our study. The reader uninterested in motivation can skip all but the definition of reach in Section \ref{critraddef}. The reader interested in knowing  more about the history and applications
of positive reach is referred to the excellent survey by Thale \cite{Thale-reach}, or Chapter 7 of \cite{Chazalbook}, which discusses reach in the context of TDA.

Section \ref{sec:GPs} defines Gaussian processes on manifolds and associated notions such as the induced metric. It also introduces the constant $\sigma^2_c(f)$. Much of this section is a quick summary of the material in \cite{Adler} needed for this paper, and once this is done we have everything defined well enough to state the main result of the paper.

The real work starts in Section \ref{critradcomp}, in which we develop  specific  representations for the critical radius of a $S^{k-1}$ embedded manifold which  form the basis of all that follows. Some of the results here already exist in the literature, and the proofs of these are relegated to an appendix. Some are new and full proofs are given in situ. Section \ref{4lemmas}  lists four lemmas, from which, together with the representation  of  Section \ref{critradcomp}, the proof of the a.s.\ convergence in the main theorem follows easily.  Following a brief section  devoted to notation,  Sections \ref{sec:lemma1}--\ref{errproof}, which is where the hardest work lies,  then prove these lemmas, one at a time. In Section \ref{sec:fluctuation} we turn to the fluctuation result of \eqref{intro:weakconv}, both proving it and describing the 
limit process. Two technical appendices complete the paper.

\section{Critical radius and positive reach}
\label{sec:importance}

\subsection{The definition}
\label{critraddef}

Throughout the paper our underlying manifold $M$ will satisfy the following assumptions:

\begin{assumption}
\label{man:ass}
$M$ is an $m$-dimensional manifold, compact, boundaryless,  oriented, $C^3$,  and connected.
 \end{assumption}
Sometimes we shall assume that $M$ is associated with a 
Riemannian metric $g$, and sometimes that it is embedded in a smooth Riemannian manifold $(\widehat{M},\hg)$.  The main example that we shall need for this paper for an embedding space is the unit sphere $\hM=S^{k-1}$, but we shall also meet the simple Euclidean case $\hM=\real^k$ when discussing tube formulae below. In the first example, geodesics are along great circles, and the associated Riemannian distance is measured via angular distance. In the second, the geometry is the standard Euclidean one.

As an aside, we note that all our results could be extended to the case of manifolds with boundary, and even stratified manifolds satisfying the kind of side conditions endemic to \cite{Adler}. However, then we would also have to suffer through all the heavy notation endemic to \cite{Adler}, which seemed unnecessary, given that our primary motivation was to establish a general principle rather than the most general  result possible.

For the main result of the paper, all of the conditions in Assumption \ref{man:ass} are required. This is not true for some of the lemmas along the way, but for ease of exposition we shall generally adopt all the conditions throughout the paper.   For the fluctuation result, we shall even need that $M$ is $C^6$, and we will add that assumption when needed. Of course, if the majority of the authors were topologists rather than probabilists, we would  probably just have assumed that $M$ is `smooth' (i.e.\ $C^\infty$) and then not have been concerned with optimal levels of differentiability.

We need the standard exponential map (cf.\  \cite{Lee})  that maps tangent vectors to  points on the manifold. This, 
for $x\in M$ and $X\in T_x\wM$, the tangent space to $x$ in $\wM$, is given by the local diffeomorphism
\beqq
\exp^{\widehat{M}}_{x}(X) \ =\ \gamma_{x,\eta_X}(\|X\|),
\eeqq
where 
$\gamma_{x,\eta_X}$ is the unit speed geodesic in $\wM$  starting at $x$ in the direction $\eta_X\stackrel{\Delta}{=}X/\|X\| \in\ S(T_x\widehat{M})$,
 the (sphere of) unit tangent vectors at $x$. 
 The notion of  reach is closely related to the radius of the largest ball around the origins  in $T_x\wM$,  $x\in M$,  for which all  the exponential maps  are, in fact,  diffeomorphisms.

To give a more formal definition, let $d_M(x,y)$  ($d_{\widehat M}(x,y)$)  denote geodesic distance between points $x,y\in M$  
($\in \widehat M$), and for $x\in M$ and $A\subset M$ set
\beqq
d_M(x,A) \ \definedas \inf_{y\in A} d(x,y),
\eeqq
with a similar definition for $x\in \widehat M$ and $A\subset \widehat M$.

Then the local reach, or local critical radius, 
of $M$ in $\widehat{M}$ at $x$, in a  direction $\eta\in S(T_x\widehat{M})$,  is defined by 
\begin{equation}
\theta_{\ell}(x,\eta)\ \definedas \ \sup\left\{p\: d_{\widehat{M}}\left(\exp^{\widehat{M}}_x(p\,\eta),M\right)=p\right\}.
\label{locrad}
\end{equation}
Thus, if $p>\theta_{\ell}(x,\eta)$, there is a point $y\neq x$ in $M$ which is closer to $\exp^{\widehat{M}}_{x}(p\,\eta)$ than $x$ is. The local critical radius of $M$ in $\widehat{M}$ at the point $x$ is defined as
\begin{equation}
\theta(M,x)\ \equiv \  \theta (x)\stackrel{\Delta}{=}\inf_{\eta\in T_x^\perp M \cap S(T_x\widehat{M})}\theta_{\ell}(x,\eta),
\label{locx}
\end{equation}
where $T_x^\perp M$ is the normal space at $x$,  of $M$ in $\widehat{M}$.
Taking  an infimum over the entire manifold finally gives the global reach, or  critical radius, of $M$ in $\widehat{M}$:
\begin{equation}
\tau(M) \equiv \tau  \ \equiv \ \theta (M,\hM )\ \stackrel{\Delta}{=}\ \inf_{x\in M}\theta(x).
\label{globc}
\end{equation}

A more picturesque definition of reach, in the Euclidean setting  for which $\widehat M=\real^n$, which also explains the terminology `critical radius' is as follows: Imagine rolling a ball of radius $r$ and  dimension $n$ over the manifold $M$, but in such a way that the ball only touches $M$ at a single point. The largest choice of radius that allows this is the critical radius.

For some examples in which $\hM$ is a Euclidean space of codimension of least one with respect to $M$, note that if $M$ is a convex set, then its reach will be infinite. In fact,  infinite reach characterizes convex sets in this case. If $M$ is a sphere, then its reach is equal to its radius. If $M$ is the disjoint union of two spheres, the reach is the minimum of the two radii and half of the closest distance between the spheres.

If $\hM$ is itself a sphere, and $M$ a great circle, then the reach of $M$ (in angular coordinates) will be $\pi/2$. In general, the reach of a closed subset of a sphere will be no more than $\pi/2$.

This is all you need to know about reach to skip to Section \ref{sec:GPs} and read the rest of the paper.
 The rest of this section is motivational.

\subsection{Medial axis}
\label{subsec:medial}
An alternative way to think of reach is via the notions of the  medial axis of $M$ and its local feature size,  notions which have been developed in the Computational Geometry community. Given $M$ embedded in $\hM$, define the set  
\beqq
G=  \left\{y\in \hM\: \exists\ x_1\neq x_2\in M \text{ such that } d_\hM (y,M) = d_\hM (y,x_1) = d_\hM (y,x_2)\right\}.
\eeqq
The closure of $G$ is called the medial axis, and for any $x\in M$ the local feature size $s(x)$ is the distance of $x$ to the medial axis. It is easy to check that
\beqq
\theta(M,\hM)\ = \ \inf_{x\in M} s (x).
\eeqq

\subsection{On tube formulae}
\label{subsec:tube}
As mentioned earlier, the   birthplace of the notion of reach is 
Weyl's volume of tubes formula,  a classical result in Differential Geometry, and an extension of the much earlier  Steiner's tube formula for convex sets in $\real^n$. Interestingly, Weyl's original paper  \cite{Weyl} was motivated by 
a question raised by Hotelling \cite{HOT} related to the  derivation of  confidence regions around regression curves. Both of these papers still make for fascinating (but not easy) reading today, and both generated enormous literatures, one mathematical (e.g.\ \cite{Gray}) and one statistical (e.g.\ \cite{Johansen} and the literature referenced there). For its importance to Probability see, for example, \cite{Adler} and the references therein.

Restricting ourselves to the Euclidean setting for the moment, define the tube of radius $\rho >0$ around  $M$ in $\widehat M =\real^k$ to be 
\beqq
\text{Tube}(M,\rho) \ = \ \left\{x\in\real^k\: \inf_{y\in M} \|y-x\|\leq \rho\right\}.
\eeqq
Then Weyl's tube formula  states that, for $\rho<\theta (M,\real^k)$,
\beq
\label{eq:tube}
\text{Vol}
\left(\text{Tube}(M,\rho)\right)=\sum_{j=0}^m\mathcal{L}_j(M)\omega_{k-j}\rho^{k-j},
\eeq
where Vol is $k$-dimensional Lebesgue volume, $\omega_{n}$ denotes the volume of a  unit $n$-dimensional ball, and the 
$\mathcal L_j{(M)}$ are the Lipschitz-Killing curvatures of $M$. These are also known as quermassintegrales, Minkowski,
 Dehn and Steiner functionals, and intrinsic volumes, although in many
of these cases the indexing  and normalisations differ. It is worth noting, as Weyl established in what he considered the  part of \cite{Weyl}  that required more than ``what could have been accomplished by any student in a course of calculus'', that  these functionals are intrinsic. That is, they are independent of the embedding of $M$ into $\real^k$. (See for example, Lemma 10.5.1 in \cite{Adler}, where this  fact is given a probabilistic proof  in the notation we use here.)

It is hard   to overstate the importance of \eqref{eq:tube}, along with its variants for  more general ambient spaces. The fact that the formula ceases to hold for $\rho$ larger than the reach means that all the applications of tube formulae also fail at some point, and it is knowing where this point is that makes the reach such an important parameter of a manifold.

\subsection{Condition numbers, manifold learning and learning homology}

Standard manifold learning scenarios usually start with a `cloud' of points $\mathcal X=\{x_1,\cdots,x_n\}$ in some high dimensional space, which are believed to be sampled from an underlying manifold $M$ of much lower dimension $m$, with or without additional noise. (Additional noise will mean that the points need not lie on $M$ itself, but rather are sampled from  some region near $M$.) A classical problem is to construct a set which approximates  $M$ is a useful fashion. This is a well known problem with a vast literature, and `useful' here is usually taken to be mean physical closeness in some norm. 

More recently, a new literature has appeared, motivated by ideas from Algebraic Topology, 
in which the aim of physical closeness  is replaced with the aim of correctly recovering  the topology of $M$. Two of the earliest papers in this area are \cite{Niyogi1,Niyogi2} (but see also \cite{ChazalL}) and  it is these papers that were in fact the original motivators of the current one.

In \cite{Niyogi1} the setup is that of a random sample from $M$, and the recovery method -- or at least the theorems describing its properties -- relies on knowing the reach $\tau$ of $M$. In this case, choosing an
$\ep\in (0, \tau/2)$, the simple union of $\ep$-balls centered at the points of $\mathcal X$ is chosen as the estimator of $M$.
That is 
\beq
\label{union:eq}
M_{estimate} \ \definedas \bigcup_{x\in\mathcal X} B_\ep (x).
\eeq
The arguments in \cite{Niyogi1} follow the two stage structure of Smale described above. Firstly, it shows   that if one has a dense enough subset of points in $M$ then $M$ is a deformation retract of $M_{estimate}$, and so both sets have the same homology. For the second stage,  it shows that if a large enough sample is taken then one can bound, from below, the probability of the sample being dense enough. The final result is that, for all small $\delta$, if 
\beqq
n\ >\ \beta_1\left(\log(\beta_2)+\log\left(\frac{1}{\delta}\right)\right),
\eeqq
where
\beq
\label{betas:eq}\quad
\beta_1=\frac{4^m\, \text{vol}(M)}{\omega_m(\ep\cos \gamma_1)^m},\qquad
\beta_2=\frac{8^m\, \text{vol}(M)}{\omega_m(\ep \cos\gamma_2)^m },
\eeq
and  
$\gamma_1=\arcsin (\ep /8\tau)$ and 
$\gamma_2=\arcsin (\ep/16\tau)$, then 
the homology of $M_{estimate}$ equals the homology of $M$ with probability at least $1-\delta$. A corresponding  result for the case of sampling with noise is given in \cite{Niyogi2}.

We have brought the above equations to show, explicitly,   how the reach appears in the complexity of this estimation problem. The smaller the  reach, is, the smaller one is forced to take $\ep$, and the larger the sample size $n$ needs to be for a given estimation accuracy.

Of course, for a given problem, one does not know what $M$ is, and so, a fortiori, little is known about its  reach. Consequently, in the spirit of Smale's two step procedure, we need to enrich the second stage by also averaging over possible $M$. The current paper is a step in this direction, by formulating a class of random Gaussian manifolds and beginning a study of their reach.

Moreover, the main result of the paper has an immediate application in the manifold learning situation. Although Theorem \ref{maintheorem} relates only to a very specific random embedding of $M$ into a high dimensional sphere, a liberal interpretation of the theorem
 implies that the part of the  complexity of the estimation problem depending on reach is more or less independent of any embedding of $M$ into a higher dimensional space. The import of this is that there is no `curse of dimensionality', related to reach, that involves the dimension of the ambient space.

Of course, we can only make these claims for the Gaussian-embedding that we study, but the fact that they are proven in the Gaussian case will alleviate concerns among practitioners, in general, that ambient dimension has an effect on reach. This was not known until now, even for a special case.

A second practical implication of this paper is the introduction, albeit implicit, of a new class of smooth random manifolds that are both reasonable and mathematically tractable. Recalling the two stage paradigm of Smale above, it would be interesting, and probably useful, to introduce into the TDA setting the notion of Bayesian optimization. In terms of the above homology learning example, by this we mean minimizing not the probability of correctly identifying the homology for a fixed (but unknown) $M$, but rather minimizing the expectation of some cost function of this probability, averaged over a (random) family of possible $M$. 
The calculations of the current paper, along with those of \cite{Krishnan} which address issues of the asymptotic isometry of Gaussian-embedded Riemannian manifolds, show that the model introduced here allows for tractable mathematical manipulation.

%In closing this discussion, it is also probably useful to note that while the Physics literature is already rich in models for what there are called random manifolds,  these are typically for use in constructive quantum field theory, and so highly  non-smooth. See, for example, the important papers of  Taubes \cite{Taubes1,Taubes2} regarding the construction of Gaussian measures on the space of all maps from a smooth manifold into a topological space. These are superficially related to our model, although the underlying 
%issues are rather different.

%%%%%%%%%%%%%%%%%%%
\section{Gaussian processes on manifolds, and the main theorem}
\label{sec:GPs}

As mentioned earlier, our basic reference for Gaussian processes is \cite{Adler}. Here we shall only give the very minimum in definitions and notation needed for this paper.

\subsection{Gaussian processes on Riemannian manifolds}
We start, as usual, with a $C^3$ compact manifold $M$, with or without an associated Riemannian metric $g$. (For the novice, Section \ref{sec:notation} explains these terms and some of the following notation.)

A real valued Gaussian process, or random field, $f\:M\to\real$, with zero mean (assumed henceforth) is then determined by its covariance function $\C \:M\times M\to\real$ given by
\beqq
\C(x,y)\ = \ \E\{f(x)f(y)\}.
\eeqq
If $\C$ is smooth enough, the process also induces a Riemannian metric on the tangent bundle $T(M)$ of $M$ defined by 
 \begin{equation}
g_x(X,Y)\ \definedas\ \E \{(Xf)(x)\times (Yf)(x)\}\  =\ Y_y X_x \C(x,y)\big|_{y=x},
\label{indmet}
\end{equation}
where $X,Y$ are vector fields with values $X_x,Y_x$ in the tangent space $T_x M$. We shall assume throughout that $\C$ is positive definite on $M\times M$, from which it follows that $g$ is a well defined  Riemannian metric, which we call the {\it metric induced by $f$}. 

From now on, we shall make one of two -- essentially complementary -- assumptions:

\begin{assumption}
\label{indmet:assump}
If, in the above setting, we are given a manifold $M$ as in Assumption \ref{man:ass} and a Gaussian process $f\:M\to\real$, but no metric on $M$, we shall assume that $M$ is endowed with the metric induced by $f$.

If, on the other hand, we start with a Riemannian manifold $(M,g)$, then we shall choose a Gaussian process in such a way that the metric  induced by \eqref{indmet} is precisely $g$. 
\end{assumption}
The fact that given a metric $g$ there exists a Gaussian process inducing this metric is a  consequence of the Nash embedding theorem (cf.\ proof of Theorem 12.6.1 in \cite{Adler}).  

The only additional assumptions that we require  relate to smoothness and non-degeneracy for $f$, but for this we need some notation.  Thus we write, from now on, $\nabla$ for the Levi-Civita connection of  $(M,g)$, 
and $\nabla^2$ for the corresponding covariant Hessian. 
 Fix an orthonormal (with respect to $g$) frame field 
$E=(E_1,\dots,E_m)$ in $T(M)$. The specific choice of $E$  is not important.
\begin{assumption}
\label{f:assump}
We assume that the zero mean Gaussian process $f\:M\to\real$ has, with probability one, continuous first, second, and third order derivatives over $M$, and, for each $x\in M$, the joint distributions of the $(1+m+m(m+1)/2)$-dimensional random vector
\beqq
\left( f(x),\, ((\nabla f)(E_i))(x), \, ((\nabla^2 f)(E_i,E_j))(x), \ 1\leq i,j \leq m \right)
\eeqq
are nondegenerate.

We shall also assume that $\E\{f^2(x)\}$, the variance of $f$,   is constant, and for convenience, we take the constant to be one. No other homogeneity assumptions are required.
\end{assumption}

Regarding Assumption \ref{f:assump}, we note that the requirement that $f\in C^3(M)$ is probably 
not necessary. It   arises as a side issue in a tightness argument in Section \ref{tight:sec}, which requires a uniform bound on increments of fourth order derivatives of $\C$. A (much) more complicated argument would probably  require only  that 
$f\in C^{2+\epsilon}$ for some $\epsilon>0$,
but rather than lose sight of the forest for the trees we are happy to live with the extra smoothness. In fact, in order to prove the fluctuation result \eqref{intro:weakconv}, we shall even have to assume that $f\in C^6(M)$. We shall explain how the need for these high levels of smoothness arise in a moment, when we have the requisite notation.

\subsection{The parameter $\sigma^2_c(f)$}
\label{sigmac2}
Given the above setting, we now define a new Gaussian 
process on 
\beq
\label{tM:def}
\tM\ \stackrel{\Delta}{=} \ (M\times M)\setminus \text{diag}(M\times M)
\eeq
% as follows: $M\times M$ 
by setting
\begin{equation}
f^x(y)\ =\   
%\begin{cases}
 \frac{f(y)-\E \left\{f(y)\,\big| \, f(x),\nabla f(x)\right\}}{1-\C(x,y)}.
 %&\text{if } x\neq y,\\
%f(x), &\text{if } x= y.
%\end{cases} 
\label{newpr:eq}
\end{equation}
The fact that this process is well defined is not obvious, since as $y\to x$ in \eqref{newpr:eq} both numerator and denominator approach zero. Nevertheless, as we shall show in Section \ref{limit:subsec},  if we have enough smoothness for $f$, then the limit behaves well. 
 For example, just to be certain that $\lim_{y\to x} f^x(y)$ is well defined requires that $f\in C^2$. 

(In fact, ratios of the $0/0$ nature appear throughout the proofs, with denominators such as $1-\C (x,y)$ (as above),
 $1-\C^2 (x,y)$, and even  $(1-\C^2 (x,y))^2$, all of which are problematic as $y\to x$. For the  
the a.s.\ convergence of \eqref{intro:globalconv}, this leads to the requirement that$f\in\C^3(M)$. For the fluctuation result \eqref{intro:weakconv} we  will even need to assume that $f\in C^6(M)$. While these conditions seem, at first, rather severe, they seem to be necessary and not just a consequence of our method of proof.)

In any case, $f\in C^3(M)$ is more than enough to ensure that it makes sense to define the  function $\sigma_c(f,\cdot)$, and constant 
the $\sigma_c(f)$, as follows:
\beq
\label{sigmax:def}
\sigma^2_c(f,x) &\definedas& \sup_{y\in M\setminus x}\text{Var}\left(f^x(y)\right),\\
\label{sigmac:def}
 \sigma^2_c(f) &\definedas& \sup_{x\in M}\sigma^2_c(f,x).
\eeq

We now have everything we need to state the main result of the paper, but, first, we explain why the above two definitions are already `well known'.

 Associated with the Gaussian process $f$ are a reproducing kernel Hilbert space, $H$, and an $L_2$ space, $\cal H$, which is the completion of the span of $f$ over $M$. Writing $S(H)$ and $S(\cal H)$ for the unit spheres of these spaces, there is an isometry, $\Psi$ between $M$, when given the  metric $g$ induced by $f$, and $S(\cal H)$, determined by $\Psi (x)=f(x)$, for all $x\in M$. There are also  isometries between $S(\cal H)$ and $S(H)$, and so between $M$ and $S(H)$, the details of which can be found, for example, in Chapter 3 of \cite{Adler}, but which date back to the earliest history of Gaussian processes.
 
 It turns out that  $\sigma^2_c(f,x)$ is precisely the local reach of $\Psi(M)$ at the point $f(x)$, when  $S(\cal H)$ is considered as a submanifold of $\cal H$.   It follows immediately that $\sigma^2_c(f)$ is the corresponding global reach. Similar statements can be made about the isometric embedding of $M$ into $S(H)$, but would take longer to explain. The bottom line, however, is that both 
 $\sigma^2_c(f,\cdot)$ and $\sigma^2_c(f)$  are basic quantities inherently connected to $(M,g)$ when it is viewed via isometric  embeddings into larger spaces, and that there is a lot of  Hilbert sphere geometry lying behind the asymptotics of this paper.
 
 These observations are relatively recent. In our current notation, they can be found in Section 14.4.3 of \cite{Adler}, but see also  \cite{Takemura-equivalence} and the references therein. 
 
 The reason that  $\sigma_c(f,\cdot)$ and $\sigma_c(f)$  have been of more recent interest is that they arise  in the rigorous justification of the so-called `Euler characteristic heuristic' for approximating the distribution of the supremum of smooth Gaussian processes. In this setting, let $\chi\left(A_u(f,M)\right)$ denote the Euler characteristic of the excursion set $A_u$ of the Gaussian field $f$, defined by
$$
A_u\ \equiv \ A_u(f,M)\ =  \ \{x\in M: f(x)\geq u\}.
$$
 It has been `well known' for some decades that, at least for high levels $u \in\real$, the mean Euler characteristic provides a good estimate of the  exceedance probability, 
$\P\left\{\sup_{x\in M}f(x)\geq u\right\}$. That is, for large $u$, the difference
\beqq
\text{Diff}_{f,M}(u)\ \definedas \ \E \left\{\chi\left(A_u(f,M)\right)\right\}-\mathbb{P}\Big\{\sup_{x\in M}f(x)\geq u\Big\}
\eeqq
is small. 

Relatively recently (cf.\  \cite{Adler,Takemura-equivalence,Takemura}) this statement has been made precise. (These sources actually treat the more general setting of stratified manifolds, which requires an  additional condition of local convexity for $M$, as well as some minor side conditions on both $M$ and $f$. The definition of $\sigma^2_c(f)$ is also correspondingly changed.  See, for example, \cite{Takemura-zero} for a discussion of why local convexity is required. In fact, what is required is close to positive reach, and the reason that \eqref{diff:eq} fails for zero reach is much the same reason that tube formulae fail. But that is another story.)
%\textcolor{blue}{and with a suitably defined $\sigma_c^2(f)$ that},
In our setting, it is now known that
\beq
\label{diff:eq}
\liminf_{u\rightarrow\infty}-u^{-2}\log\left|\text{Diff}_{f,M}(u)\right|\ \geq \ \smallhalf
\left(1+\frac{1}{\sigma^2_c(f)}\right).
\eeq

\subsection{Main result}
With the  introduction,  motivation and almost all of the notation behind us, we are almost ready to state the main result of the paper. However, two more items of notation are required. The first gives the  local radius of $h^k(M)$, as  a submanifold of $S^{k-1}$, at the image, under $h^k$, of the point $x\in M$. This is given, for $x\in M$,  by 
\beq
\label{thetak:def}
\theta_k(x) \ \definedas \ \inf_{\eta\in T_{h^k(x)}^\perp(h^k(M))\cap S(T_{h^k(x)}S^{k-1})} \theta_\ell(h^k(x),\eta),
\eeq
where $\eta$ is a unit vector in the tangent space $T_{h^k(x)}S^{k-1}$ pointing in a normal direction to $h^k(M)$ at $h^k(x)\in h^k(M)$. The second gives the global reach, as 
\beq
\label{thetak:defn}
\theta_k \ \definedas \inf_{x\in M} \theta_k(x).
\eeq

\begin{theorem}
Let $M$ be a manifold satisfying Assumption \ref{man:ass}, and let  $f\:M\to\real$ be a Gaussian process satisfying {Assumptions \ref{indmet:assump} and \ref{f:assump}}.
 Assume that $\sigma^2_c(f)$, as defined by \eqref{sigmac:def}, is finite. Consider the embedding (\ref{embedding}) of $M$ into the unit sphere in $\mathbb{R}^k$, and let $\theta_k$ be the global critical radius of the random manifold $h^k(M)$. Then, with probability one,
 \beq
 \label{mainconvergence:eq}
\cot^2\theta_k\  \to \ \sigma_c^2(f), \quad \text{as } k\to\infty.
\eeq
If, in addition,  $M$ is $C^6$, and the sample paths of $f$ are a.s.\ $C^6$ on $M$, then there exists a sequence $\bar\gamma_k$ of random processes from $M\to\real$, such that, for all $x\in M$,  
\beq
\label{bound-diff}
\sqrt{k}\left|\cot^2\theta_k(x)-\sigma^2_c(f,x)\right| \ \leq \ \left|\bar\gamma_k(x)\right|,
\eeq
and a limit process $\bar\gamma\:M\to\real$, such that,
\beq
\label{flucform:sec3}
\bar\gamma_k(\cdot)  \ \Rightarrow\  \ \bar\gamma (\cdot),
\eeq
where the convergence here is weak convergence, in the Banach  space of  continuous functions on $M$ with supremum norm, and 
\beqq
\bar\gamma (x) \ = \ \sup_{y\in M\setminus x}\gamma(x,y), 
\eeqq
where $\gamma$ is the Gaussian process over $\tM$ defined by
 \eqref{gammaeq}.
\label{maintheorem}
\end{theorem}

We defer all further discussion of the fluctuation result of  \eqref{bound-diff} and \eqref{flucform:sec3}   until Section 
\ref{sec:fluctuation}, where it will be restated as  Theorem \ref{fluctheorem}, and the (rather involved) definition of the process $\gamma$ will appear. Until then we shall concentrate on the a.s.\ convergence of \eqref{mainconvergence:eq}.

As an aside, note that a variation of some of the easier arguments in the following sections show that  the sequence of mappings $h^k$ tends, with probability one, to an isometry, in the sense that the associated  pullbacks  to $M$ of the usual metric on $S^{k-1}$ tend to the  induced metric \eqref{indmet} on $M$. We provide a rigorous treatment of this result in \cite{Krishnan}, albeit with the self-normalisation of \eqref{embedding}   replaced by a $\sqrt{k}$ normalisation. We also prove there that this gives rise to the a.s.\ convergence of a class of intrinsic functionals  of $h^k(M)$ to the corresponding functionals  on $(M,g)$.   We refer you  to \cite{Krishnan} for  details.

%\textcolor{blue}{Another remark that we would like to make here is that even though reach is not an intrinsic functional of a manifold, Theorem \ref{maintheorem} states that the critical radii of $h^k(M)\subset S^{k-1}$ with the usual metric converge to the critical radius of $M$ isometrically embedded via $f(\cdot)$ into the unit sphere in the Hilbert space formed by span$(f(x):x\in M)$ (with the $L^2$ metric of course). This follows from the discussion in Section 14.4.3 of \cite{Adler}; specifically, see Lemma 14.4.3 there. In hindsight, this geometric intuition is the one that takes away the surprise of the result of Theorem \ref{maintheorem}, once it is proved. Combining this with what we mentioned in the paragraph above, this Hilbert sphere geometry is what determines the asymptotic results arising out of our model.}

\section{Computation of the Critical Radius}
\label{critradcomp}

This section contains two purely geometric lemmas from which follow the  probabilistic computations that make up most of the paper. The first gives a characterisation of the reach of general submanifolds of spheres, and the second  does the same for the specific submanifolds $h^k(M)$ in terms of the functions $f^k$. 
To start, recall that geodesic distance on the sphere is measured in terms of angles, $r\in [0,\pi)$. 
Let $M$ be a submanifold of $S^{k-1}$,  and  $\eta_x$  a unit normal vector at $x\in M$. 

 We can now state the following characterisation, which implicitly assumes, as we shall from now on, that $M$ has dimension at least one. As stated it is identical to Lemma 2.1  of \cite{Takemura-equivalence}, restricted to our setting. (\cite{Takemura-equivalence} treats the more general setting of stratified manifolds.). Furthermore, as pointed out there, the proof is essentially the same as the proof given in \cite{Johansen} for the one-dimensional case. Nevertheless, because of its importance to this paper, and (only) for the sake of completeness, we give the proof in Appendix 1.

\begin{lemma}
Let  $M$ be a submanifold of $S^{k-1}$, satisfying the conditions of Assumption \ref{man:ass}. Let $T^{\perp}_x M \subset T_x S^{k-1}$ be the normal space of $M$ at $x$ as it sits in $S^{k-1}$,  viz.\ the orthogonal complement of $\text{span}(T_x M\oplus \{x\})\subset T_x\mathbb{R}^k$ in $T_x\mathbb{R}^k$. Then the critical radius, $\theta(x)$, at $x$ is given by
\begin{equation*}
\cot^2(\theta(x))\ =\ \sup_{y{\in M\setminus \{x\}}}\frac{\|P_{T^{\perp}_x M}y\|^2}{(1-\langle x,y\rangle)^2},
\end{equation*}
where $P_{T_x^{\perp}M}$ is orthogonal projection onto $T_x^{\perp}M$.
\label{crm}
\end{lemma}

We are now in a position to derive the global reach of our random manifold $h^k(M)$. The result is given in the next lemma. However, before stating and proving the lemma, we need some preparatory definitions. 

Recalling the  embedding maps  $h^k$ and the components $f^k$ of \eqref{embedding}, let  $(X_1,..,$ $X_m)$ be a frame bundle of full rank over $M$, and define the $k\times (m+1)$ matrix
$$L_x\ =\ \begin{bmatrix}
      f_1(x) &X_1f_1(x)& \cdots&X_mf_1(x)           \\
      \vdots &  \vdots          &  \vdots& \vdots \\
       f_k(x)           & X_1f_k(x)& \cdots&X_mf_k(x)
     \end{bmatrix},$$
and the projection matrix $$P_x=L_x(L_x^ {T}L_x)^{-1}L_x^ {T}.$$
By the  independence of the $f_j$ and the non-degeneracy of Assumption \ref{f:assump}, the rows of $L_x$ are a.s.\ linearly independent, and so  
$L_x^ {T}L_x$ is a.s.\ invertible.  The matrix $P_x$
orthogonally  projects vectors in $\real^k$ onto  
 \beqq \text{span}\left(h^k(x),\,h^k_{*}(X_i),\,1\leq i\leq m\right), 
 \eeqq
considered as a subset of $T_{h^k(x)}\mathbb{R}^k$,
where $h^k_{*}: T_x M\rightarrow T_{h^k(x)}S^{k-1}$ is the usual push forward operator.

Consider now the following expression, well known from the Statistics literature as the maximum likelihood estimate based on $k$ samples of the correlation coefficient between $f(x)$ and $f(y)$; viz. 
\begin{equation}
\widehat{\C}_k(x,y)=\frac{\sum_{j=1}^kf_j(x)f_j(y)}{\sqrt{\sum_{j=1}^k(f_j(x))^2\sum_{j=1}^k(f_j(y))^2}}.
\label{MLest}
\end{equation}
Consider the conditional process $f^x(y)$ defined on $\tM$ by \eqref{newpr:eq}
%
%Define a new process on $\tM \stackrel{\Delta}{=}(M\times M)\setminus\text{diag}(M\times M)$ as follows:
%\begin{equation}
%f^x(y)=\frac{f(y)-\E \left\{f(y)/f(x),\nabla f(x)\right\}}{1-C(x,y)}.
%\label{newpr}
%\end{equation}
and denote $k$ i.i.d.\ realizations of it at $y$ by $$f^{x,k}(y)=(f^x_1(y),\cdots,f^x_k(y)).$$
Define an `error process' 
\begin{equation}
E^{x,k}(y)=\frac{k}{\|f^k(y)\|^2}\frac{(1-\C(x,y))^2}{(1-\widehat{\C}_k(x,y))^2}\left(\frac{1}{k}\|P_x f^{x,k}(y)\|^2\right).
\label{errt}
\end{equation}
The key lemma to be proven before starting probabilistic calculations is the following.

\begin{lemma}
Let  $M$ be a manifold satisfying the conditions of Assumption \ref{man:ass}, embedded into $S^{k-1}$ via the embedding map defined in (\ref{embedding}). Assume that $f$ satisfies Assumptions \ref{indmet:assump}  and \ref{f:assump}. Then, with probability one, the reach  of $h^k(M)$ is given by
\begin{equation}
\cot^2 \theta_k
=\sup_{x\in M}\sup_{y\in M\setminus \{x\}}\frac{k}{\|f^k(y)\|^2}\frac{(1-\C(x,y))^2}{(1-\widehat{\C}_k(x,y))^2}\left(\frac{1}{k}\|f^{x,k}(y)\|^2\right)-E^{x,k}(y).
\label{subeq}
\end{equation}
\label{lma:implemma}
\end{lemma}

\begin{proof}
The global critical radius is obtained by taking infima of local critical radii as given in (\ref{globc}). However, since the cotangent is decreasing in the first quadrant, we have
$$\cot^2\theta_k\ =\ \sup_{x\in M}\cot^2\theta_k(x).$$
Using the result from Lemma \ref{crm}, the above is equal to
\beq
\sup_{x\in M}\sup_{y\in M\setminus \{x\}}\frac{\|(I-P_x)h^k(y)\|^2}{\left(1-\langle h^k(x),h^k(y)\rangle\right)^2}.
\label{star}
\eeq
Since $f$ is centered Gaussian, its derivatives are also centered Gaussians. Thus, the orthogonal projection $P_x(f(y))$ of $f(y)$ onto the space spanned by $f(x)$ and $\nabla f(x)$ is 
$$\E \left\{f(y)\,\big|\, f(x),\nabla f(x)\right\}.$$ 
This  observation, along with  \eqref{newpr:eq},  \eqref{star}, and the fact that $(I-P_x)f^x(y)=f^x(y)$, show that 
\beqq
\cot^2\theta_k\ = \ \sup_{x\in M}\sup_{y\in M\setminus \{x\}}\frac{k}{\|f^k(y)\|^2}\frac{(1-\C(x,y))^2}{(1-\widehat{\C}_k(x,y))^2}\left(\frac{1}{k}\|(I-P_x)f^{x,k}(y)\|^2\right).
\eeqq
From the fact that we have orthogonal projections, this is 
$$\sup_{x\in M}\sup_{y\in M\setminus \{x\}}\frac{k}{\|f^k(y)\|^2}\frac{(1-\C(x,y))^2}{(1-\widehat{\C}_k(x,y))^2}\left(\frac{1}{k}\|f^{x,k}(y)\|^2\right)-E^{x,k}(y),$$
and the  lemma is proven.
\end{proof}

We shall see later that the error term $E^{x,k}(y)$ in  \eqref{mainequ} goes to zero, and so we shall be primarily concerned 
with the  
%uniformly  subeq}. As already stated, we are interested in obtaining a strong law for $\cot^2\theta(h^k)$ as $k\rightarrow \infty$.
%In view of the comment that the error $E^{x,k}(y)$ uniformly converges to zero (proved in Section \ref{errproof}), we basically need to focus on the first component in (\ref{subeq}). Thus, from now on, we study 
the a.s.\ convergence of
\begin{equation}
\sup_{x\in M}\sup_{y\in M\setminus \{x\}}\frac{k}{\|f^k(y)\|^2}\frac{(1-\C(x,y))^2}{(1-\widehat{\C}_k(x,y))^2}\left(\frac{1}{k}\|f^{x,k}(y)\|^2\right).
\label{mainequ}
\end{equation}
For this, we need to establish convergence results for the three terms here. 
The results  we need are stated as four lemmas in the next section.

\section{Four key lemmas and the proof of the main theorem}
\label{4lemmas}
The proof of Theorem \ref{maintheorem} follows from the four lemmas stated below and is given at the end of this section. Throughout this section we shall assume, without further comment, that $M$ satisfies the conditions of Assumption \ref{man:ass}. The conditions on $f$ vary, since not all the lemmas require the same level of smoothness. All the conditions, however, are implied by Assumptions  \ref{indmet:assump}  and \ref{f:assump}.

We start by showing that the first two terms in (\ref{mainequ}) converge uniformly,  with probability one, to $1$.

\begin{lemma}
Let $f^k$ be a $\real^k$-valued random process on $M$,  with i.i.d.\ components, each a centered, unit variance Gaussian process  over $M$, with a.s.\ continuous sample paths. Then, with probability one, 
$$\lim_{k\rightarrow\infty}\sup_{y\in M} \left|\frac{k}{\|f^k(y)\|^2}-1\right|\ =\ 0.$$
\label{firstterm}
\end{lemma}

\begin{lemma}
Let $f^k$ be as in the Lemma \ref{firstterm}, but also $C^3$. Denote the covariance function  of its components by  $\C(x,y)$, 
and let $\widehat{\C}_k(x,y)$ be as defined in (\ref{MLest}).  Then, with probability one,
$$\lim_{k\rightarrow\infty}\sup_{(x,y)\in \tM}\left|\left(\frac{1-\C(x,y)}{1-\widehat{\C}_k(x,y)}\right)^2-1\right|\ =\ 0.
$$
\label{secondterm}
\end{lemma}

The third lemma (after some trivial calculations) will -- see below -- give us that the remaining term  in (\ref{mainequ})   converges to the 
parameter $\sigma^2_c(f)$.
\begin{lemma}
Under the same assumptions as in Lemma \ref{secondterm}, and with probability one,
\beq
\lim_{k\rightarrow\infty}\sup_{(x,y)\in\tM}\left|\frac{\|f^{x,k}(y)\|^2}{k}-\text{\rm Var}\left(f^x(y)\right)\right|\ =\ 0.
\eeq
%where 
%$\text{Var}\left(f^x(y)\right)$ is the variance of the process over $\tM$ defined in (\ref{newpr}).
\label{thirdterm}
\end{lemma}

It will follow from the proof of this lemma that $f^x(y)$ is bounded even when we are arbitrarily close to diag($M\times M$). 
This is needed to ensure that all the terms defined in  (\ref{mainequ}) are, a priori, well defined.

The final step we need is the following.

\begin{lemma}
Under the same assumptions as in Lemma \ref{secondterm}, and with  $E^{x,k}(y)$ as defined in (\ref{errt}), we have, with probability one,
$$\lim_{k\rightarrow\infty}\sup_{(x,y)\in\tM}E^{x,k}(y)\ =\ 0.$$
\label{errorterm}
\end{lemma}

We now show how to prove the main result as a straightforward consequence of the previous four lemmas.

\begin{proof}[Proof of Theorem \ref{maintheorem}]

It is immediate from the results of Lemmas \ref{firstterm}, \ref{secondterm} and \ref{errorterm} that, with probability one,

\begin{equation}
\lim_{k\rightarrow\infty}\cot^2\theta_k\ =\ \lim_{k\rightarrow\infty}\sup_{(x,y)\in\tM}\frac{\|f^{x,k}(y)\|^2}{k},
\label{limex}
\end{equation}
and we shall be done once we show that the right hand limit  is $\sigma^2_c(f)$. 

However, this is immediate from the much stronger result in Lemma \ref{thirdterm} that  
$$\lim_{k\rightarrow\infty}\sup_{(x,y)\in\tM}\left|\frac{\|f^{x,k}(y)\|^2}{k}-\text{Var}\left(f^x(y)\right)\right|\ =\ 0$$
and that, by definition, 
\beqq
\sigma^2_c(f) \ =\  \sup_{(x,y)\in\tM} \text{Var}\left(f^x(y)\right).
\eeqq
This completes the proof of Theorem \ref{maintheorem}, modulo proving the four lemmas.
%
% is 
%$$\lim_{k\rightarrow\infty}\left|\sup_{(x,y)\in\tM}\frac{\|f^{x,k}(y)\|^2}{k}-\sigma^2_c(f)\right|=0?$$
%We have that $$\forall (x,y)\in \tM,\,\forall \ep >0,\,\exists N(\ep )\,\,\text{s.t.}\,\,\left|\frac{\|f^{x,k}(y)\|^2}{k}-\text{Var}\left(f^x(y)\right)\right|<\ep \,\,\,\,\forall k>N(\ep ),$$
%where $N(\ep )$ is an a.s. finite integer-valued random variable. Therefore,
%$$\text{if}\,\,k>N(\ep ), \sup_{(x,y)\in\tM}\text{Var}\left(f^x(y)\right)-\ep \leq \sup_{(x,y)\in\tM}\frac{\|f^{x,k}(y)\|^2}{k}\leq \sup_{(x,y)\in\tM}\text{Var}\left(f^x(y)\right)+\ep .$$
%Thus, we would have the following result as stated in Theorem \ref{maintheorem}:
%$$\lim_{k\rightarrow\infty}\cot^2\theta(h^k)=\sigma_c^2(f)\,\,\text{a.s}.$$
%
\end{proof}

\section{Some (standard) notation}
\label{sec:notation}
Many of the proofs to follow freely use standard notation from Differential Geometry. Since we expect that not all readers will be familiar with this, we include here a brief notational guide.  There are many standard texts to which one can turn for details. Lee's book \cite{Lee} is our favourite,  but the quick and dirty treatment in Chapter 7 of \cite{Adler} also suffices.

{%\color{red} 
We are working with a Riemannian manifold $(M,g)$, for which the Riemannian metric is, for each $x\in M$, an inner product $g_x$ on the tangent space $T_xM$ to $M$ at $x$. If $\{(U_\alpha,\phi_\alpha)\}_\alpha$ is an atlas for $M$, then for each chart  $(U_\alpha,\phi_\alpha)$ we shall often need a (local) orthonormal frame field 
$X^{\alpha}=\{X_1^\alpha,\dots,X^\alpha_m\}$  for the tangent bundle  $T^\alpha M\definedas \{T_xM,\ x\in U_\alpha\}$, where orthonormality is in the  metric $g$.  We shall refer to this later as ``choosing an orthonormal frame field", without specific reference to  charts or the index $\alpha$. Since all
our later definitions and calculations are local (i.e.\ can be carried out in terms of local charts) this is not a problem (and global issues such as parallelizability do not arise.)}

If $F\:M\to\real$ is a smooth function, and $X=X_x\in T_xM$, then by 
\beqq
XF(x),\ \ (XF)(x),\ \ (X_xF)(x),\ \ X_xF(x), \ \ X_xF_x,\ \ \text{etc},
\eeqq
 we mean the derivative of $F$ in direction $X_x$ at $x$. At various times we will make use of all of these possible notations, so as to make individual formulae either clear and/or compact.
 
As opposed to the above derivatives, the gradient,  $\grad F$, of $F$ 
is the unique continuous vector field 
on $M$ such that
\beq
\label{geometry:gradient:equation}
  g_x(\grad F_x, X_x)\  = \ X_xF,
\eeq
 for every   vector field $X$. If $F$ is a function of more than one parameter, say $F(x,y)$, then we will denote the gradient with respect to $x$ as $\grad_xF(x,y)$, etc.

The (covariant) Hessian $\Hess F$ 
   is the bilinear symmetric map from 
 $C^1(T(M)) \times C^1(T(M))$ to $ C^0(M)$ 
(i.e.\ it is a double differential
 form)  defined by  
\beq
\label{geometry:Hessian:equation}\qquad
 (\Hess F)(X,Y) \equiv  \Hess F(X,Y)  \definedas    XYF - \nabla_XYF =  g(\nabla_X \grad F, Y),
\eeq
 where, while $\nabla$ with no subscript denotes the gradient, when it is subscripted with a vector field, as in $\nabla_X$, it indicates the  the 
 Levi-Civita connection of $(M,g)$. 

It is standard that $\Hess F$ 
could also have been defined to be $\nabla(\nabla F)$, which is from where the notation comes. Recall that 
in the simple Euclidean case the Hessian is typically considered to be the $N\times N$ matrix 
 $H_F=(\partial^2
F/\partial x_i\partial x_j)_{i,j=1}^N$. In the more general setting above, $H_F$  defines the   two-form by setting
$\nabla^2f(X,Y)= XH_FY'$. In this case
\eqref{geometry:Hessian:equation} follows from simple
calculus.
 
 We shall need the obvious, but important, fact that if $x$ is a critical point of $F$ (i.e.\ $\nabla F(x)=0$)
 then $XF(x)=0$ for all $X\in T_xM$  and so by 
\eqref{geometry:Hessian:equation} it follows that 
$ \Hess F(X,Y)(x)= XYF(x)$. Consequently, at these points, the  
Hessian  is independent of the metric $g$.

This concludes our brief excursion into notation. We can now turn to the proofs of the four lemmas of Section \ref{critradcomp}.

\section{Proof of Lemma \ref{firstterm}} 
\label{sec:lemma1}
We need to prove that
\beq
\lim_{k\rightarrow\infty}\sup_{y \in M}\left|\frac{k}{\|f^k(y)\|^2}-1\right|\ =\ 0,\qquad \text{a.s}.
\label{five:eq}
\eeq

However, this follows almost trivially from the following standard strong law for  Banach space valued random variables, which, since we use it often, we state in full.

\begin{theorem}[\cite{Talagrand-Ledoux}, Corollary 7.10]
Let $X$ be a Borel random variable with values in a separable Banach space $B$ with norm $\|\cdot\|_{\text{B}}$. Let $S_n$ be the partial sum of $n$ i.i.d. realizations of $X.$ Then, $$\frac{S_n}{n}\stackrel{\text{a.s.}}{\longrightarrow} 0$$ if, and only if, $\E \{\|X\|_{\text{B}}\}<\infty$ and $\E \{X\}=0$.
\label{t1}
\end{theorem}

To prove \eqref{five:eq}, we set  $X=f^2(y)-1$ in the above theorem. The Banach space $B$ is $C(M)$ (continuous functions over $M$), equipped with the sup norm. The mean zero condition is trivial. To check the moment condition on the norm of $f^2-1$, note that
\beqq
\E \left\{\sup_{y\in M}|f^2(y)-1|\right\} &\leq& 1+\E \left\{\sup_{y\in M}|f^2(y)|\right\}
\\ &\leq& 1+\E \left\{\left(\sup_{y\in M}|f(y)|\right)^2\right\}<\infty.
\eeqq
Finiteness of the expectation here follows from the Borell-Tsirelson-Ibragimov-Sudakov inequality (e.g.\ Theorem 2.1.2 in \cite{Adler}).  This is all that is needed to prove \eqref{five:eq}.

%That is, the existence of even exponential moments (as shown there) is guaranteed. The a.s. convergence of the reciprocal now follows from limit laws.
%\begin{equation*}
%\therefore \lim_{k\rightarrow\infty}\sup_{\stackrel{(x,y)\in M\times M}{y\neq x}}\left|\frac{k}{\|f^k(y)\|^2}-1\right|=0\,\,\text{a.s}.
%\label{1conv}
%\end{equation*}

\section{Proof of Lemma \ref{thirdterm}}
\label{LemPf:sec}
Before starting this proof in earnest, we need  to check that all the terms that are implicitly assumed to exist in the statement of the lemma are well defined. In particular, we need to consider the limits
\beq
\label{limdefined:eq}
%\sup_{(x,y)\in\tM}\frac{\|f^{x,k}(y)\|^2}{k}.
\lim_{y\to x} f^x(y) \ = \ \lim_{y\to x} \frac{f(y)-\E \left\{f(y)\,\big| \, f(x),\nabla f(x)\right\}}{1-\C(x,y)},
\eeq
the problem being that both numerator and denominator tend to zero in the limit. If \eqref{limdefined:eq} is not well defined, then the supremum in the lemma makes no sense. (Note that away from the diagonal in $M\times M$ there is no problem with either boundedness or continuity, due to the assumed smoothness of $f$.)

\subsection{The limit \eqref{limdefined:eq} is well defined}
\label{limit:subsec}

The proof is basically an application of  L'H\^opital's rule. To start, we take an orthonormal frame field 
$X=\{X_1,\dots,X_m\}$  for the tangent bundle of $M$, where orthonormality is in the induced metric $g$ of \eqref{indmet} and with the conventions described in Section \ref{sec:notation}.

Then standard computations for this situation (cf.\ Section 12.2.2 of \cite{Adler} for precisely this case)  give that the vector $(f(y),f(x),\nabla f(x))$ has a mean zero, multivariate Gaussian distribution with covariance matrix
\beqq
\left[
\begin{matrix}
1 & \C(x,y) & \nabla_x\C(x,y) \\
\C(x,y)  & 1 & 0\\
\nabla_x \C(x,y)  & 0 & 1
\end{matrix}
\right].
\eeqq
%(Note that, in general, the covariance of  $X_i f(x)$ is $-\nabla^2_x \C(x,y)(X_i,X_j$. However, since we are assuming that 
%$f$ has constant variance, and that all computations are taken in the induced metric, it follow that $\grad^2_x\C(x,x)i,X_j)=\delta$ for $i\neq j$,  (again, see Section 12.2.2 of \cite{Adler}), and so from the definition \eqref{newpr:eq} of $f^x(y)$ and 
From this and the definition of Gaussian conditional expectations, it immediately follows that   
%\begin{equation}
%f^x(y) = f(x)+\frac{f(y)-f(x)}{1-\C(x,y)}-\left(\sum_{i=1}^m\frac{ {X}_i f(x)\,{X}_i \C(x,y)}{\nabla^2 \C(x,x)({X}_i,{X}_i)}\right)\frac{1}{1-\C(x,y)}.
%\label{eq:fxyrep}
%\end{equation}
\begin{equation}
f^x(y) = f(x)+\frac{f(y)-f(x)}{1-\C(x,y)}-\sum_{i=1}^m\frac{ {X}_i f(x)\,{X}_i \C(x,y)}{1-\C(x,y)}.
\label{eq:fxyrep}
\end{equation}

%However, since all our calculus is being done with respect to the Riemannian metric \eqref{indmet}, the complicated denominator $\nabla^2 \C(x,x)({X}_i,{X}_i)$ above is identically one. (Recall that $\nabla$ is the Levi-Civita connection associated with the process induced $g$). 

Now take any $X=\sum_{i=1}^m d_i X_i \in T_xM$, and let $c$ be a $C^2$ curve in $M$ such that
$$c\: (-\delta,\delta)\rightarrow M, \quad c(0)=x,\quad  \dot{c}(0)=X.$$ 
As $y\to x$ along this curve, we  have
\beq
\label{limita:eq}\qquad 
\lim_{y\rightarrow x}f^x(y)=\lim_{u\rightarrow 0}\left[f(x)-\frac{\left(\frac{ {d}f(c(u))}{{d}u}-\sum_{i=1}^m X_i f(x)\frac{ {d}X_i \C(x,c(u))}{ {d}u}\right)}{\frac{ {d}\C(x,c(u))}{ {d}u}}\right].
\eeq

Consider the limit of the ratio in the above expression, this being the only problematic term. This is
\begin{equation}
\frac{(\sum_i d_iX_i)f(x)-\sum_i(X_i f(x)(\sum_j d_j(X_jX_i)\C(x,x))}{(\sum_i d_iX_i)\C(x,x)}.
\label{li}
\end{equation}
Note that because of our choice of  Riemannian metric, and the fact that the $X_i$ were chosen to be orthonormal, we have
$$X_jX_i \C(x,x)\ =\ g(X_i,X_j) \ = \  \delta_{ij},$$
the Kronecker delta. Therefore the numerator in $(\ref{li})$ is zero. The denominator is also zero because of the assumption of constant variance on $f$. Thus, to find the true limit, another application of   L'H\^opital's rule is necessary, and so we have
$$ 
\lim_{y\rightarrow x} f^x(y)=f(x)-\lim_{u\rightarrow 0}\left(\frac{\frac{ {d}^2f(c(u))}{ {d}u^2}-\sum_iX_if(x)\frac{ {d}^2X_i\C(x,c(u))}{ {d}u^2}}{\frac{ {d}^2\C(x,c(u))}{ {d}u^2}}\right).
$$
It is easy to see that
$$\lim_{u\rightarrow 0}\frac{ {d}^2f(c(u))}{ {d}u^2}=\nabla^2f(x)(X,X)+\nabla_{X}Xf(x),
$$
and
\beqq
\lim_{u\rightarrow 0}\frac{ {d}^2X_i \C(x,c(u))}{ {d}u^2}& =& XXX_i\C(x,x)\\
&=& \E \{XXf\,X_if\}  \\
&=& \E \{\nabla_XXf\,X_if\} \\
&=&g(\nabla_XX,X_i),
\eeqq
where in the second-last equality, we have used calculations from \cite{Adler} (cf. Eqn. (12.2.14)).
Consequently, 
$$ \lim_{u\rightarrow 0}\sum_iX_if(x)\frac{ {d}^2X_i \C(x,c(u))}{ {d}u^2}\ =\ \sum_ig(\nabla_XX,X_i)X_if(x)\ =\ 
\nabla_XXf(x),
$$
and so, moving to the notation of 2-forms, the limit in \eqref{limita:eq} is given by the well defined expression
$$
f(x)-\frac{\nabla^2f(x)(X,X)}{\nabla^2 \C(x,x)(X,X)},
$$ 
and the limit in \eqref{limdefined:eq}, albeit dependent on the path of approach of $y$ to $x$, is also well defined. As a consequence,  we also have that, for each finite $k$, 
$$\sup_{(x,y)\in\tM}\frac{\|f^{x,k}(y)\|^2}{k}$$
 is a.s.\ finite.

\subsection{Completing  the proof} 
We now turn to the proof of the lemma, establishing \eqref{thirdterm}.

This, however, follows exactly along the lines of the proof of Lemma \ref{firstterm}, again applying Theorem \ref{t1}. We need only take as our Banach space $C_b(\tM)$, the bounded, continuous functions on $\tM$ with supremum norm, and as our basic random variable $X=(f^{x}(y))^2-\text{Var}\left(f^x(y)\right)$.

The previous subsection establishes the a.s.\  boundedness of $X$ needed to make the argument work.

%
%
% 
%Make the definitions:
%$$\sigma^2_c(f,x)=\sup_{y\in M\setminus x}\text{Var}\left(f^x(y)\right);\,\,\sigma^2_c(f)=\sup_{x\in M}\sigma^2_c(f,x).$$ Since $f$ is centered Gaussian, for each $x$, $f^x(y)$ is also a centered Gaussian. Since the limit $$\lim_{y\rightarrow x}f^x(y)$$ exists, it is also a.s. bounded over $\tM$. Set $X=(f^{x,k}(y))^2-\text{Var}\left(f^x(y)\right)$ in Theorem \ref{t1}, and take $B$ as the Banach space of bounded, continuous functions over $\tM$ with $\|\cdot\|_{\infty}$ norm. We need to ask whether
%$$\E \left\{\sup_{(x,y)\in\tM}\left|(f^x(y))^2-\text{Var}\left(f^x(y)\right)\right|\right\}<\infty?$$
%Assume $\sigma_c^2(f)<\infty.$ Then,
%$$\E \left\{\sup_{(x,y)\in\tM}\left|(f^x(y))^2-\text{Var}\left(f^x(y)\right)\right|\right\}\leq \E \left\{\left(\sup_{(x,y)\in\tM}\left|f^x(y)\right|\right)^2\right\}+\sigma_c^2(f)<\infty,$$
%the finiteness following from Borell-TIS inequality (please refer Theorem 2.1.2 in \cite{Adler}). Now, Theorem \ref{t1} gives that
%$$\lim_{k\rightarrow\infty}\sup_{(x,y)\in\tM}\left|\frac{\|f^{x,k}(y)\|^2}{k}-\text{Var}\left(f^x(y)\right)\right|=0\,\,\text{a.s.}$$

\section{Proof of Lemma \ref{secondterm}}
\label{sec:lemma3}

Lemma \ref{secondterm} involves showing that the ratio
\beqq
\frac{1-\C(x,y)}{1-\widehat{\C}_k(x,y)}
\eeqq
converges, uniformly,  to one, as $k\to\infty$. For given $x\neq y$, this is straightforward, following from a strong law of large numbers, much as in the previous two proofs. However, as $x\to y$, even for fixed $k$,  there is no easy way to 
find a uniform bound on the ratio, since both numerator and denominator tend to zero.

\subsection{Outline of the proof}
\label{sec:outhard}
We start by writing $\wC$ as a sum of three terms:
$$
\widehat{\C}_k(x,y)\ =\ \C(x,y)+\text{Bias}(\widehat{\C}_k (x,y))+\xi_k(x,y),
$$
where $\xi_k(x,y)$ is mean zero, random error with variance 
$\text{Var}(\widehat{\C}_k)$, and the deterministic bias term is $\E\{\wC_k -\C\}$.

We shall show in Appendix 2 that 
\beq
\text{Bias}(\widehat{\C}_k(x,y))&=&\frac{-\C(x,y)(1-\C^2(x,y))}{2{(k+1)}}+O\left(\frac{1}{k^2}\right),   \label{bias:eq}
\\ \text{Var}(\widehat{\C}_k(x,y))&=&\frac{(1-\C^2(x,y))^2}{k}+O\left(\frac{1}{k^2}\right).
\label{variance:eq}
\eeq
and that  the remainder terms  in both expressions are uniformly bounded over  $ M\times M$. (In fact, this is almost classical, in that
expressions for the bias and variance of the correlation coefficient estimator centered around the sample means (as opposed to $\widehat{\C}$, which is centered at zero) are well known in the Statistics
literature, dating back, at least, to \cite{Kendall}[{Chapter 16, see (16.73) and (16.74)}]. Appendix 2 treats the zero-centered $\widehat{\C}$ case.)

For notational convenience, set 
\beq
\label{tzeta:eq}
 \tzeta_k(x,y)\  \definedas \ \sqrt{k}\left(\widehat{\C}_k(x,y)-\C(x,y)\right)
%\ =\ \sqrt{k}\left(\xi_k(x,y)+\text{Bias}(\widehat{\C}_k(x,y))\right).
\eeq
Since the notation is getting long, from now on,  we interchangeably use $a_k(x,y)$ and $a_k^{xy}$ for a function $a_k$ of $x$ and $y$, refrain from writing out explicitly the summation indices and their range in some situations where they are obvious, and also introduce
\beqq
\beta_k^{xy} \ \definedas \ \text{Bias} (\wC_k(x,y)).
\eeqq
Then, in view of \eqref{bias:eq}, and up to a term of $O(k^{-2})$ in the denominator,  we have
\beq
\label{long:eq}
&&\left(\frac{1-\C(x,y)}{1-\widehat{\C}_k(x,y)}\right)
\\ && \qquad\qquad= \
\frac{1-\C^{xy}}{1-\C^{xy}+\frac{\C^{xy}(1-(\C^{xy})^2)}{2{(k+1})}-\sqrt{k}\frac{\left(\widehat{\C}_k^{xy}-\C^{xy}-\beta_k^{xy}\right)}{1-(\C^{xy})^2}\frac{1-(\C^{xy})^2}{\sqrt{k}}}.  \notag
\eeq
Cancelling $(1-\C^{xy})$ from numerator and denominator,  this becomes

\beq
\label{pretoprove}
\left({1+\frac{\C^{xy}(1+\C^{xy})}{2{(k+1)}}-\sqrt{k}\frac{\left(\widehat{\C}_k^{xy}-\C^{xy}-\beta_k^{xy})\right)}{1-(\C^{xy})^2}\frac{1+\C^{xy}}{\sqrt{k}}}\right)^{-1}.
\eeq
The only problematic term here is 
\begin{equation}
\sqrt{k}\frac{\left(\widehat{\C}_k(x,y)-\C(x,y)-\beta_k(x,y))\right)}{1-\C^2(x,y)},
\label{toprove}
\end{equation}
since the second term converges deterministically to zero, and the final multiplicative factor is bounded by $2/\sqrt{k}$.  We shall prove that the sequence of random processes defined by  \eqref{toprove}  converges weakly to a continuous  Gaussian process on $\tM$. This, the extra divisor of $\sqrt{k}$ in \eqref{pretoprove}, and some elementary probability arguments which we leave to the reader, will be enough to prove Lemma  \ref{secondterm}.

In fact, in view of \eqref{bias:eq}, we can drop the bias term from \eqref{toprove}, and suffice with the weak convergence, over $\tM$,
 of the processes
\beq
  \zeta_k(x,y)   \ \definedas \   \frac{ \tzeta_k(x,y)}{1-\C^2(x,y)}  \ = \  \sqrt{k}
  \frac{\widehat{\C}_k(x,y)-\C(x,y)}{1-\C^2(x,y)}   .
 \label{zetak:eq}
\eeq
We shall prove this in a number of stages.

To start, we show the weak convergence of the numerator in \eqref{zetak:eq} -- $\tzeta_k$ -- which is much less delicate than that of the ratio $\zeta_k$, there being no $0/0$ issues. The convergence of the finite dimensional distributions is shown in 
the following Section \ref{wkzeta:sec} and the tightness in \ref{tight:sec}. The final step is to apply the continuous mapping theorem, (e.g.\ \cite{Billingsley}, Section 1.5) for which we need to know that the mapping between function spaces that takes
\beq
\label{mapping:eq}
\phi(x,y)\ \to \ \frac{\phi(x,y)}{1-\C^2(x,y)}  
 \eeq 
is continuous, with probability one, for $\tzeta$, the process on $M\times M$ which is the limit of the $\tzeta_k$. We have already seen that ratios like that on the right hand side here are problematic, and computable, at the $y\to x$ limit, only via L'H\^opital's rule. Consequently, the weak convergence of the $\tzeta_k$ is going to have to be in a function space with a norm that takes into 
account convergence of derivatives as well as the function values. This is going to make the tightness argument rather intricate, which is why Section \ref{tight:sec} is the longest in the paper.  The continuous mapping argument will be given at the end, in the brief Section \ref{cmt:sec}.

\subsection{Fi-di convergence of $\tzeta_k$, and characterising the limit}
\label{wkzeta:sec}

The main result of this section is the following.

\begin{lemma}
The finite-dimensional distributions of $ \tzeta_k$, on $M\times M$,  converge to those of the zero mean, 
 Gaussian process, $\tzeta$, with covariance function given by
\beq \notag
\E \{ \tzeta(x_0,y_0) \tzeta(x,y)\}
&=&
\smallhalf {\C^{x_0 y_0}\C^{xy}} \left[(\C^{y_0 x})^2+(\C^{y_0 y})^2+(\C^{x_0 x})^2+(\C^{x_0 y})^2\right]\\
&&\quad
+\C^{xy_0}\left[\C^{x_0 y}-\C^{x_0 x}\C^{xy}\right]+\C^{y_0 y}\left[\C^{x_0 x}-\C^{x_0 y}\C^{xy}\right]
\notag 
\\
&&\qquad -\C^{x_0 y_0}\left[\C^{x_0 x}\C^{x_0 y}+\C^{y_0 y}\C^{xy_0}\right]. 
\label{correlationfunction}
\eeq
\label{fidi:lemma}
\end{lemma}

The proof will rely on the following result of Anderson.

\begin{theorem}[\cite{Anderson}, Theorem 4.2.3]
Let $\{U(k)\}$ be a sequence of $d$-component random vectors and $b$ a fixed vector such that $\sqrt{k}(U(k)-b)$ has the limiting distribution $\mathcal{N}(0,T)$ as $k\rightarrow\infty$. Let $g(u)$ be a vector-valued function of $u$ such that each component $g_j(u)$ has a nonzero differential at $u=b$, and let $\psi_b$ be the matrix with  $(i,j)$-th component $({\partial g_j(u)}/{\partial u_i})|_{u=b}$. Then $\sqrt{k}(g(u(k)-g(b))$ has the limiting distribution $\mathcal{N}(0,\psi_b^\prime T\psi_b)$.
\label{Anderson:thm}
\end{theorem}

{\smc Proof of Lemma \ref{fidi:lemma}} \ \ 
As one might guess from the complicated form of \eqref{correlationfunction} the  calculations involved are somewhat tedious,
 and so we shall concentrate on making the main steps clear. 
 %The idea is to use Theorems 3.4.4 and 4.2.3 from \cite{Anderson}. 
 Towards that end, we introduce the following notation just for this proof. 
 For any $i,j\in\mathbb{N}$, and points $(x_i,y_j)\in M\times M$, define
$$    \bar{C}^{x_i ,y_j}_{11}(k)=\| f^k(x_i)\|^2,\,\,    \bar{C}^{x_i, y_j}_{22}(k)=\| f^k(y_j)\|^2,\,\,    \bar{C}^{x_i ,y_j}_{12}(k)=\sum_{\ell=1}^k f_\ell(x_i)f_\ell(y_j).$$ 
Now define
\beqq
U(k) &=& \frac{1}{k}  %\begin{bmatrix}
\left(     \bar{C}_{11}^{x_1,y_1},   %\\[0.35em]
    \bar{C}_{22}^{x_1,y_1},   %\\[0.35em]
    \bar{C}_{12}^{x_1,y_1},  %\\[0.35em]
\dots,   %\\[0.35em]
    \bar{C}_{11}^{x_n,y_n},   %\\[0.35em]
    \bar{C}_{22}^{x_n,y_n},   %\\[0.35em]
    \bar{C}_{12}^{x_n,y_n}\right), \\
%\end{bmatrix}, \,\,
b &=& \left( 1, 1,    %\\[0.35em]
  \C^{x_1,y_1},      %\\[0.35em]
\dots,         %\\[0.35em]
1,     %\\[0.35em]
1,      %\\[0.35em]
  \C^{x_n,y_n}  \right).
%\end{bmatrix}.$$
\eeqq
Thus, the elements of $U$ are the maximum likelihood  estimators of the corresponding elements of $b$. It then follows from standard estimation theory (e.g.\ \cite{Anderson},   Theorem 3.4.4)  that $\sqrt{k}(U(k)-b)$ has a limiting normal distribution with mean $0$ and some  covariance matrix $T$, the specific structure of which does not concern us at the moment.
% at the moment. (we do not intend to show the detailed derivation of the exact components of $T$). 

%
%$$\sqrt{k}\left(\frac{\sum_{j=1}^k f_j(x) f_j(y)}{\sqrt{\sum_{j=1}^k(f_j(x))^2}\sqrt{\sum_{j=1}^k(f_j(y))^2}}-\C(x,y)\right)= \tzeta_k(x,y)%$$

In order to prove the lemma, we require the asymptotic distribution of 
\beq
\label{conv:rob1}
\left\{\sqrt{k}(\widehat{\C}_k(x_i,y_i)-\C(x_i,y_i))\right\}_{i=1}^n.
\eeq
However, using the vector $U$ above it is easy to relate the $\wC$s to the $\bar {C}$s, and if we now 
define a function $g\:\mathbb{R}^{3n}\rightarrow\mathbb{R}^n$ by
$$g(u_1,u_2,\cdots,u_{3n})=\left(\frac{u_3}{\sqrt{u_1 u_2}},\frac{u_6}{\sqrt{u_4 u_5}},\cdots, \frac{u_{3n}}{\sqrt{u_{3n-1} u_{3n-2}}}\right),$$ 
then Theorem \ref{Anderson:thm} establishes the claimed convergence of finite dimensional distributions, and so proves the lemma, modulo two issues.

The first is the condition on the differential required by Theorem \ref{Anderson:thm}, but this is trivial. The second is to derive the exact form \eqref{correlationfunction} of the limiting covariances, which, while not intrinsically hard, is a long and tedious calculation. The calculation starts 
by writing out the covariance function for   $\wC$ and computing moments, all of which involve Gaussian variables. Fortunately, most of the detailed calculations that we need were carried out long ago  in the statistical literature and, can be found, for example,  in \cite{Kendall3}[e.g.\ Chapter 41, Example 41.6]. What remains is to send $k\to\infty$ in these expressions, and deduce \eqref{correlationfunction}. We shall not go through the tedious details here.
\qed

\subsection{Tightness  of $ \tzeta_k$}
\label{tight:sec}

For the reasons alluded to above and exploited below, we shall prove tightness in  the Banach space of twice continuously differentiable functions on $M\times M$, which we denote by $\Btwo$, equipped with the norm 
\beq
\label{Bnorm:defn}  %\\
\qquad \ 
\|f\|_\Btwo  \definedas   \max\left\{ \|f\|_{\infty},\, \|\nabla f\|_{\infty},\, \|\nabla^2 f\|_{\infty}\right\},
%\notag
\eeq
where the norms of the first and second order derivatives are obtained by taking maximum over the norms of the $2m$ and $4m^2$ components of the Riemannian differential and Hessian, respectively. 

To break the rather long proof of tightness  into bite sized pieces, we write
\beq
\tzeta_k(x,y) \ = \  \alpha_k(x,y)\, +\, \Delta_k(x,y), 
\label{pieces:eq}
\eeq
where 
\beqq
\alpha_k(x,y)\ \definedas \
\sqrt{k}\left(\frac{\frac{1}{k}\sum_{j=1}^k (f_j(x) f_j(y)\, - \, \C (x,y))}{\sqrt{\frac{\sum (f_j(x))^2}{k}}\sqrt{\frac{\sum (f_j(y))^2}{k}}}\right)
\eeqq
and
\beqq
\Delta_k(x,y) \ \definedas \ \sqrt{k}\left(1-\sqrt{\frac{\sum (f_j(x))^2}{k}}\sqrt{\frac{\sum (f_j(y))^2}{k}}\right)
\frac{ \C(x,y)}{\sqrt{\frac{\sum (f_j(x))^2}{k}}\sqrt{\frac{\sum (f_j(y))^2}{k}}}.
\eeqq

In the following two subsections we shall prove that the sequences $\alpha_k$ and $\Delta_k$ converge weakly in $\Btwo$, from which the convergence of $\tzeta_k$ immediately follows.

\subsubsection{$\alpha_k$ converges weakly in $\Btwo$}
\label{Sec:alpha}
We start with something even simpler than $\alpha_k$, viz.\ the sequence of random functions $\eta_k$ defined by
\begin{equation}
\eta_k(x,y) \ = \ \sqrt{k}\left(\frac{1}{k}\sum f_j(x) f_j(y)\, -\, \C(x,y)\right).
\label{know}
\end{equation}
To prove the weak convergence of this sequence, we use the theorem stated as part of Example 1.5.10 in \cite{van}[p41] (also see the discussion after the statement of the theorem). 

To this end, note that the summands in \eqref{know} are
 i.i.d.\ copies of the random function $f\otimes f\:M\times M\to\real$
defined by $(f\otimes f) (x,y) = f(x)f(y)$. If we endow $M\times M$ with the topology induced by the Riemannian distance $d_{M\times M}$ (this is the metric we use in place of the semi-metric in the theorem from \cite{van}), then  $M\times M$ is compact in this topology. Since the  convergence of the finite dimensional distributions of (\ref{know}) follows from  Theorem \ref{Anderson:thm}, all that is left to check for weak convergence
of   \eqref{know} is  tightness.

In order to show tightness, we first need to set up some notation, in particular  for Taylor expansions on
$M\times M$,   in terms of Riemannian normal coordinates.

Consider normal neighbourhoods $U_1,U_2$ (local coordinates $(x^i),(y^i),$ respectively) around $x_0$ and $y_0$ in $M$, and take $U_1\times U_2$ around $(x_0,y_0)$ in $M\times M$. Then, $(x^i:y^i)$ give us the following definition of coordinates in the product space $U_1\times U_2$:
\beqq
(x^1,\dots, x^m:y^1,\dots, y^m)(x_0,y_0)\ = \
\left[(x^1, \dots, x^m)(x_0):(y^1,\dots, y^m)(y_0)\right].
\eeqq
 Since $$T_{(x_0,y_0)}(U_1\times U_2)=T_{x_0}U_1\oplus T_{y_0}U_2,$$any tangent vector $v$ in the product tangent space splits uniquely as the sum of $v_1 \in T_{x_0}U_1$ and $v_2 \in T_{y_0}U_2$. This further gives us the following definition for the exponential map in $M\times M$:
$$\exp_{(x_0,y_0)}^{M\times M}(v)=(\exp_{x_0}^M(v_1),\exp_{y_0}^M(v_2)).$$ 
Let the coordinate basis vectors for the tangent spaces be denoted by $\left( \frac{\partial}{\partial x^i}\right)$ and $\left( \frac{\partial}{\partial y^i}\right)$,  considered as row vectors. The concatenation of the two serves as a basis for the product tangent space, and so any vector $v$ in this space can be written as
\beqq
v\ = \ v_1\oplus v_2 \ = \ 
%^1,\dots,v^{2m}) \= \ 
\sum_{i=1}^mv^i\frac{\partial}{\partial x^i}+\sum_{i=1}^{m}v^{i+m}\frac{\partial}{\partial y^i}.
\eeqq
%
%
%
%the first part being the coordinate representation of $v_1$, the second, $v_2$. Now, say along $v$, we move to the point $(x,y)$ from $(x_0,y_0)$. Then, along $v'=v_1+0,$ we move to $(x,y_0)$ from $(x_0,y_0)$, and along $v''=0+v_2$, we move to $(x_0,y)$ from $(x_0,y_0)$ (note that these are equivalent to motions in $M$ from $x_0$ to $x$ and $y_0$ to $y$, respectively). If $v\approx 0$, we have 

This allows us to write the following Taylor expansion for $\C(x,y)$ about $(x_0,y_0)$:
\beqq
\C(x,y) &=& \C(x_0,y_0)\, +\, v\left[\left(\frac{\partial \C(x,y)}{\partial x^1},\dots,\frac{\partial \C(x,y)}{\partial y^m}\right)|_{(x_0,y_0)}\right]^T  \\
&& \qquad\qquad\qquad +\
\smallhalf v\left[ \frac{\partial^2\C(x,y)}{\partial^k x_i \partial^l y_j}|_{(x_0,y_0)}\right]
v^T
 \ +\ O(\|v\|^3),
\eeqq
where $k+l=2$. Finally, we recall a few important facts from the topic of normal coordinates and geodesics (cf. \cite{Lee}) that the geodesic starting from $(x_0,y_0)$ in the direction $v$ is given in Riemannian normal coordinates by $t(v^1\cdots v^{2m})$, geodesics are locally minimizing, and so along with the previous fact, we have $\|v\|^2=d^2_{M\times M}((x,y),(x_0,y_0))$. Also, importantly, since Christoffel symbols vanish at the centers of the normal charts, covariant derivatives at the centers reduce to  usual partial derivatives. Therefore, working with normal coordinates is useful in local calculations.

Returning now to the issue of tightness, we need to establish moment bounds on the second derivatives of the processes $\eta_k$ of 
\eqref{know}.
%\beq
%\label{eta:def}
%\eta_k(x,y)\ \definedas \
%\sqrt{k}\left( \frac{ \sum f_j(x) f_j(y)}{k}-\C(x,y)\right).
%\eeq
Clearly, the variance {and} correlation function of $\eta_k$, are, respectively, the variance of $f(x)f(y)-\C(x,y)$ and the correlation 
$$E\{(f(x)f(y)-\C(x,y))(f(x_0)f(y_0)-\C(x_0,y_0))\}.$$

To investigate second derivatives, it is useful to move to the notation of 2-forms. Doing so, it follows from simple algebra that   the diagonal elements of the Hessian matrix of this process are given by
\beqq
&&\bigg\{\nabla^2 f(x)\left(\frac{\partial}{\partial x^i},\frac{\partial}{\partial x^i}\right) f(y)-\nabla^2 \C^{xy}\left(\frac{\partial}{\partial x^i},\frac{\partial}{\partial x^i}\right),  \\
&&\qquad \quad \nabla^2 f(y)\bigg(\frac{\partial}{\partial y^j},\frac{\partial}{\partial y^j}\bigg) f(x)-\nabla^2 \C^{xy}\left(\frac{\partial}{\partial y^j},\frac{\partial}{\partial y^j}\right)\bigg\}, \quad {1\leq i,j\leq m},
\end{eqnarray*}
with other elements in the upper triangular portion falling into one of the three groups 
$$\nabla^2 f(x)\left(\frac{\partial}{\partial x^i},\frac{\partial}{\partial x^j}\right) f(y)-\nabla^2 \C^{xy}\left(\frac{\partial}{\partial x^i},\frac{\partial}{\partial x^j}\right),\qquad1\leq i<j\leq m,$$
$$\nabla^2 f(y)\left(\frac{\partial}{\partial y^i},\frac{\partial}{\partial y^j}\right) f(x)-\nabla^2 \C^{xy}\left(\frac{\partial}{\partial y^i},\frac{\partial}{\partial y^j}\right),\qquad 1\leq i<j\leq m,$$ or
$$\frac{\partial f(x)}{\partial x^i}\frac{\partial f(y)}{\partial y^j}-\nabla^2 \C^{xy}\left(\frac{\partial}{\partial x^i},\frac{\partial}{\partial y^j}\right),\qquad1\leq i\leq j\leq m.$$ 
For the sake of illustration, we focus on only one `type' of term. Computations for the other terms
are basically the same. The term we shall consider is
$$\nabla^2 f(x)\left(\frac{\partial}{\partial x^1},\frac{\partial}{\partial x^2}\right) f(y)-\nabla^2 \C^{xy}\left(\frac{\partial}{\partial x^1},\frac{\partial}{\partial x^2}\right),$$
and we now also note that the term involving the derivatives of $\C$ does not present any problem for the upper bound on the increments since $\C$ is deterministic and  the fact that $f\in C^3$ implies that $\C$ is at least $C^6$. Thus, it is enough to prove the following bound 
\begin{eqnarray*}
E\left\{\left(\nabla^2 f(x)\left(\frac{\partial}{\partial x^1},\frac{\partial}{\partial x^2}\right) f(y)-\nabla^2 f(x)\left(\frac{\partial}{\partial x^1},\frac{\partial}{\partial x^2}\right)|_{x=x_0} f(y_0)\right)^2\right\}\\
\leq K d^2_{M\times M}((x,y),(x_0,y_0))
\end{eqnarray*}
for any two pairs $(x,y), (x_0,y_0)\in M\times M$, and constant $K$. 
The expectation here is bounded  above  by
\beq
\label{upperbound}
&&\ \ 2E\left\{\left(\frac{\partial^2 f(x)}{\partial x^2\partial x^1}f(y)-\frac{\partial^2 f(x)}{\partial x^2\partial x^1}|_{x=x_0}f(y_0)\right)^2\right\}
\\
&& \quad+2E\left\{\left(\left(\nabla_{\frac{\partial}{\partial x^1}}\frac{\partial}{\partial x^2}f(x)\right) f(y)-\left(\nabla_{\frac{\partial}{\partial x^1}}\frac{\partial}{\partial x^2}f(x)\right)|_{x=x_0} f(y_0)\right)^2\right\}.
\nonumber
\eeq
As far as the first  expectation here is concerned, using Wick's formula, the fact that $f$ has unit variance, 
 and, for a differential  operator  $D$ of any order, writing   $D\C^{x_0y_0}$  for $D\C^{xy}|_{(x_0,y_0)}$, we have that it is equal to 
 \beq
&& \left[\frac{\partial^4 \C^{xx}}{\partial (x^2)^2\partial (x^1)^2}+\frac{\partial^4 \C^{x_0x_0}}{\partial (x^2)^2\partial (x^1)^2}-2\frac{\partial^4 \C^{xx_0}}{\partial (x^2)^2\partial (x^1)^2}\C^{yy_0}\right]
\label{expectation} \\
&&\qquad\qquad +\ \left[\left(\frac{\partial^2 \C^{xy}}{\partial x^2\partial x^1}\right)^2+\left(\frac{\partial^2 \C^{x_0y_0}}{\partial x^2\partial x^1}\right)^2-2\frac{\partial^2 \C^{xy_0}}{\partial x^2\partial x^1}\frac{\partial^2 \C^{x_0y}}{\partial x^2\partial x^1}\right].
\nonumber
\eeq
The important point to be checked is that terms which are $O(1)$ and $O(\|v\|)$ cancel. We check this thoroughly below.  The second order terms can be trivially bounded using the facts that $f\in C^3$  and $|v_i|\leq \|v\|.$ This technique of bounding gives the required constant $K$ independent of the points, but does not offer too much insight. Consequently, we illustrate  how  this can be done for one case only.

Consider the case 
\beqq
&&\left[\frac{\partial^3 \C^{xx_0}}{\partial (x^1)^3}|_{x=x_0}v^1
+\cdots+\frac{\partial^3 \C^{xx_0}}{\partial x^m\partial (x^1)^2}|_{x=x_0}v^m\right]\\
&&\qquad\qquad \times \left[\frac{\partial^3 \C^{yy_0}}{\partial (y^1)^3}|_{y=y_0}v^{m+1}+\cdots+\frac{\partial^3 \C^{yy_0}}{\partial y^m\partial (y^1)^2}|_{y=y_0}v^{2m}\right].
\eeqq
The above is obviously smaller than
\begin{eqnarray*}
&&\left[\ \left|\frac{\partial^3 \C^{xx_0}}{\partial (x^1)^3}|_{x=x_0}v^1\right| +
\cdots+\left|\frac{\partial^3 \C^{xx_0}}{\partial x^m\partial (x^1)^2}|_{x=x_0}v^m\right|\ \right]\\
&&\qquad\qquad\times\left[\ \left|\frac{\partial^3 \C^{yy_0}}{\partial (y^1)^3}|_{y=y_0}v^{m+1}\right|+\cdots+\left|\frac{\partial^3 \C^{yy_0}}{\partial y^m\partial (y^1)^2}|_{y=y_0}v^{2m}\right|\ \right].
\end{eqnarray*}
This immediately  yields the following upper bound, in which $M_3$ is a bound on the third order derivatives 
of $\C$:
\beqq
&& M_3^2[|v^1|+\cdots+|v^m|]\times[|v^{m+1}|+\cdots+|v^{2m}|]
\\ && \qquad\qquad \leq\  M_3^2m^2d^2_{M\times M}((x,y),(x_0,y_0))  \ = \  M_3^2m^2\|v\|^2  .
\eeqq
%The trivial bound $|v^i|\leq \|v\|$ gives 
%$$M_3^2m^2\|v\|^2=M_3^2m^2d^2_{M\times M}((x,y),(x_0,y_0))$$
%as the final upper bound (with $K=M_3^2m^2$).

With the second order terms out of the way, we return to our claim that the  zeroth  and first order terms cancel out in (\ref{expectation}). Focus first on the second term in that expression. For any $(x,y)\in M\times M,$ introduce the function
$$g(x,y)\ \definedas \  \frac{\partial^2 \C^{xy}}{\partial x^2\partial x^1},$$
which, by  assumption, is at least $C^4$. Expanding this in a Taylor series about $(x_0,y_0)$, we have
\begin{equation*}
g(x,y)=g(x_0,y_0)+\sum_{i=1}^m \frac{\partial g(x,y)}{\partial x^i}|_{(x_0,y_0)}v^i+\sum_{i=1}^m \frac{\partial g(x,y)}{\partial y^i}|_{(x_0,y_0)}v^{m+i}+O(\|v\|^2).
\end{equation*}
In shorter notation, let us write the above as
\begin{equation*}
g(x,y)\ =\ g(x_0,y_0)+\sum_{i=1}^{2m}e_i v^i+O(\|v\|^2),
\end{equation*}
where the $e_i$ are the coefficients from the previous formula. Then,
\begin{equation}
\left(\frac{\partial^2 \C^{xy}}{\partial x^2\partial x^1}\right)^2+\left(\frac{\partial^2 \C^{x_0y_0}}{\partial x^2\partial x^1}\right)^2
\ =\ 
2(g(x_0,y_0))^2+2g(x_0,y_0)\sum_{i=1}^{2m}e_i v^i+O(\|v\|^2).
\label{fpart}
\end{equation}
Next, define the following smooth functions over $M$:
$$h(x)\ \definedas \ \frac{\partial^2 \C^{xy_0}}{\partial x^2\partial x^1},\qquad t(y)\ \definedas \ \frac{\partial^2 \C^{x_0y}}{\partial x^2\partial x^1}.$$
It is immediate that $$h(x)=g(x_0,y_0)+\sum_{i=1}^m e_i v^i+O(\|v\|^2),\,\,t(x)=g(x_0,y_0)+\sum_{i=m+1}^{2m}e_i v^i+O(\|v\|^2).$$
Therefore,
\begin{equation}
-2\frac{\partial^2 \C^{xy_0}}{\partial x^2\partial x^1}\frac{\partial^2 \C^{x_0y}}{\partial x^2\partial x^1}=-2[g^2(x_0,y_0)+g(x_0,y_0)\sum_{i=1}^{2m}e_iv^i]+O(\|v\|^2).
\label{spart}
\end{equation}
It is now clear that, as claimed,  at least for the second expression in  (\ref{expectation}),  the zeroth and first order  terms in the Taylor expansion cancel (cf.\ (\ref{fpart}) and (\ref{spart})).

Turning now to the first term in (\ref{expectation}), define a function $w\:\text{diag}(M\times M)\to\real$ by
$$w(x,x)\ \stackrel{\Delta}{=}\ \frac{\partial^4 \C^{xx}}{\partial (x^2)^2\partial (x^1)^2},$$
and a function on $a\:M\to\real$ by
$$a(x)\ \stackrel{\Delta}{=} \ \frac{\partial^4 \C^{xx_0}}{\partial (x^2)^2\partial (x^1)^2}.$$
These  admit the Taylor series  expansions
$$w(x,x)\ =\ w(x_0,x_0)+2\sum_{i=1}^m\frac{\partial w(x,x)}{\partial x^1}|_{(x_0,x_0)}v^i+O(\|v\|^2),$$
and
$$a(x)\ =\ w(x_0,x_0)+\sum_{i=1}^m\frac{\partial w(x,x)}{\partial x^1}|_{(x_0,x_0)}v^i+O(\|v\|^2).$$
Noting that the Taylor series expansion of $\C^{yy_0}$ about $y_0$ is
$$1+O(\|v\|^2),$$
it is easy to see that here  also, only the terms starting from the second order remain. Again,  following the same basic lines as in the previous argument shows that a similar upper bound holds for the second expectation in (\ref{upperbound}). From our earlier discussions, we are therefore done regarding proof of tightness of (\ref{know}).

In addition, since by  Lemma \ref{firstterm}, we know that $\sqrt{\sum (f_j(x))^2/k}$ converges, uniformly over $M$, and with probability one, to 1, we have  (e.g.\ \cite{Billingsley}[Theorem 4.4]) the joint weak convergence of the pair
\beqq
\left(\sqrt{k}\left(\frac{\sum f_j(x) f_j(y)}{k}-\C(x,y)\right),\ \sqrt{\frac{\sum (f_j(x))^2}{k}}\sqrt{\frac{\sum (f_j(y))^2}{k}}\right).
\eeqq

Given this, the  continuous mapping theorem immediately yields the weak convergence of $\alpha_k$, as required.

\subsubsection{$\Delta_k$ converges weakly in $\Btwo$}
\label{Sec:932}

Recall the expression for $\Delta_k$:
\begin{equation}
\sqrt{k}\left(1-\sqrt{\frac{\sum (f_j(x))^2}{k}}\sqrt{\frac{\sum (f_j(y))^2}{k}}\right)\times \frac{\C(x,y)}{\sqrt{\frac{\sum (f_j(x))^2}{k}}\sqrt{\frac{\sum (f_j(y))^2}{k}}}.
\label{diff}
\end{equation}
We have already seen that the denominator in the rightmost ratio here converges a.s., and uniformly, to one, and so a simple 
application of {Theorem 4.4 from \cite{Billingsley} and} the continuous mapping theorem imply that we need only concern ourselves with the weak convergence of the sequence
of processes $\Gamma_k$ defined by
\beq
\Gamma_k(x,{y}) \ \definedas \ {\sqrt{k}\left(\sqrt{\frac{\sum (f_j(x))^2}{k}}\sqrt{\frac{\sum (f_j(y))^2}{k}}-1\right)}. 
\eeq

To prove this convergence, we shall, for large enough $k$, bound $\Gamma_k$ from above and below by two sequences of processes, 
which converge to the same limit.
These bounds (cf.\ \eqref{final:inequ} below)  involve a common term, the weak convergence of which is known, and a smaller term, which converges a.s.\ and uniformly to zero.

The bound depends on the following  algebraic  inequality, due to Cartwright and Field \cite{Cartwright-Field}.

\begin{theorem}[\cite{Cartwright-Field}]
Let $w_i,\,1\leq i\leq n$ be numbers summing to 1. Let $x_i$ be positive numbers in $[a,b]$ ($0<a<b$), whose arithmetic and geometric means are denoted by $AM_w$ and $GM_w$, respectively. Then, 
$$\frac{1}{2b}\sum w_i(x_i-AM_w)^2\ \leq\  AM_w-GM_w\ \leq\  \frac{1}{2a}\sum w_i(x_i-AM_w)^2.$$
\label{CF:theorem}
\end{theorem}

To apply Theorem \ref{CF:theorem}  we note first that we know that  $k^{-1} \sum (f_j(x))^2$ converges to $1$ a.s.\ and
 uniformly.   Thus, given any $\ep>0$, there 
 exists a (random) $k_0$ such that, for all  $ x\in M$, and all $k\geq {k_0}$,
 \beqq
 1-\ep  \leq \frac{1}{k} \sum f^2_j(x)  \   \leq \  1+ \ep.
  \eeqq 
 
Now apply the theorem, assuming that $k\geq k_0$, taking   $n=2$, $[a,b]=[1-\ep ,1+\ep ]$,  $w_1=w_2=1/2$,  and
\beqq
x_1 \ = \   \sum f_j^2(x)/k, \quad
x_2 \ = \  \sum f_j^2(y)/k.
\eeqq
Setting
\beq
\label{Nk:firsttime}
 N_k(x) \ \definedas \  \sqrt{k}\left( \sum f_j^2(x)/k\, - \, 1\right) ,
 \eeq
a little algebra leads to 
\beq
\label{final:inequ}
&&\smallhalf \left( N_k(x) + N_k (y)\right) - \frac{(N_k(x) -N_k(y))^2}{4(1-\ep)\sqrt{k}}
\\ &&\qquad\qquad  \leq \ \Gamma_k(x,y)\ \leq \ \smallhalf \left( N_k(x) + N_k (y)\right) - \frac{(N_k(x) -N_k(y))^2}{4(1+\ep)\sqrt{k}} .
\nonumber
\eeq
 
 But this is precisely the inequality that we described above.  Although it would be straightforward to establish it independently, the weak convergence of $N_k$ has already been proven in Section \ref{Sec:alpha}, since $N_k$ is just the  process (\ref{know}) over $\text{diag}(M\times M)$. From this immediately follows the weak convergence of the processes from $M\times M\to \real$ defined
by $(x,y)\to N_k(x)+N_k(y)$ and  $(x,y)\to N_k(x)-N_k(y)$.   

This completes the proof of the weak convergence of $\Delta_k$.

\subsection{The continuous mapping argument}
\label{cmt:sec}
To complete the proof of Lemma \ref{secondterm}, we exploit the  fact, proven in the previous two sections, that $\tzeta_k$ converges weakly in $\Btwo$ to the Gaussian process $\tzeta$ with covariance function
given by \eqref{correlationfunction}, and use it to show that the ratio processes
\beq
\label{zetak:again}
\zeta_k(x,y) \ = \ \frac{\tzeta_k (x,y)}{1-\C^2(x,y)}
\eeq
converge weakly in $C_b(\tM)$.

As described at the beginning of this section, this follows immediately from an application of the continuous mapping theorem, once we show that the mapping $H\:\Btwo\to C_b(\tM)$, defined by
\beq
\label{mappingGin:eq}
(H(\phi ))(x,y)\ = \ \frac{\phi (x,y)}{1-\C^2(x,y)}  
 \eeq 
is continuous, with probability one, for the probability measure supported on the paths of $\tzeta$. 
%
%\begin{theorem}[cf. \cite{Billingsley}]
%Suppose $h$ is a measurable mapping between two metric spaces $S$ and $S'$, and $D_h$ is the set of discontinuities of $h$. Let $X, X_i$ be random elements of $S$. 
%If $X_n\implies X$ and $P(X\in D_h)=0,$ then $h(X_n)\implies h(X).$
%\label{ctsmap}
%\end{theorem}
% 

Recall that $\tzeta$ is at least $C^2$ over $M\times M$ because of weak convergence in $\Btwo$. The same (and more) is true for the covariance function $\C$, so the issue of continuity of $H$ is trivial if we restrict $\gamma$ to a region away from the diagonal of $M\times M$.

So the only question remaining is what happens as $y\to x$. What we shall now do is investigate the limits
\beqq
\lim_{y\to x} \frac{ \tzeta(x,y)}{1-\C^2(x,y)},
\eeqq
and show that they depend only on ratios of well defined functions of the second derivatives of 
$\tzeta$ and $\C$. This will immediately imply the continuity of $H$, and thus complete the proof of Lemma
\ref{secondterm}.

To this end, take $X_x\in T_x M,$ and a $C^2$ curve $c$ in $M$ such that
\beqq
c\: (-\delta,\delta)\rightarrow M,\quad c(0)=x,\quad \dot{c}(0)=X_x.
\eeqq
Then, using the symmetry of $\tzeta$ and $\C$,
\beqq
\lim_{y\to x}\frac{ \tzeta(x,y)}{1-\C^2(x,y)}\ =\ \lim_{u\rightarrow 0}\frac{ \tzeta(c(u),x)}{1-\C^2(c(u),x)}.
\eeqq
It follows from \eqref{correlationfunction} that the limit of the  numerator is zero, while the same is true of the denominator since $\C(x,x)=1$. 

By L'H\^opital's rule, the limit above is equal to
\beq
\label{gettingthere:eq}
 \lim_{u\rightarrow 0}\frac{\frac{ {d} \tzeta(c(u),x)}{ {d}u}}{-2\C(c(u),x)\frac{ {d}\C(c(u),x)}{ {d}u}},
\eeq
The denominator here is easily seen to be  zero, since $x=y$ is a critical point for $\C(x,y)$ and $\C$ is differentiable.
To check that the same is true for the numerator, note that $\tzeta$ is differentiable, with zero mean and  covariance
function given by the second derivative of the covariance function of $\tzeta$. That is, 
\beqq
\left[\E \left\{(X_x \tzeta(x,y))^2\right\}\right]_{y=x}  \ =\ \left[X_{y_1}X_{y_2}\E \left\{ \tzeta(x,y_1) \tzeta(x,y_2)\right\}\right]_{y_1=y_2=x}.
\eeqq

Using the specific form   \eqref{correlationfunction} of this covariance,  and denoting $X_{ x_1}\C^{ x_1 x}$
for  $X_{ y_1}\C^{ y_1 x}|_{ y_1= x_1}$,  we have that
\beqq
&&X_{ y_1}\E \{ \tzeta( y_1, x) \tzeta( y_2, x)\}|_{ y_1= x_1}\\
&&\quad = \smallhalf {(\C^{ y_2 x})^3} X_{ x_1}\C^{ x_1 x}+\smallhalf {\C^{ y_2 x}} X_{ x_1}\C^{ x_1 x}+
{\mbox{$\frac{3}{2}$}}{\C^{ y_2 x}}(\C^{ x_1 x})^2X_{ x_1}\C^{ x_1 x}\\
&&\qquad +\smallhalf {\C^{ y_2 x}} \left[(\C^{ x_1 y_2})^2X_{ x_1}\C^{ x_1 x}+2\C^{ x_1 x}\C^{ x_1 y_2}X_{ x_1}\C^{ x_1 y_2}\right]+\C^{ y_2 x}X_{ x_1}\C^{ x_1 x}
\\   &&\quad\qquad   -(\C^{ y_2 x})^2X_{ x_1}\C^{ x_1 y_2}
+X_{ x_1}\C^{ x_1 y_2}-\C^{ y_2 x}X_{ x_1}\C^{ x_1 x}  \\
 &&\qquad\qquad   -2\C^{ x_1 y_2}\C^{ x_1 x}X_{ x_1}\C^{ x_1 x}-(\C^{ x_1 x})^2X_{ x_1}\C^{ x_1 y_2}-\C^{ y_2 x}X_{ x_1}\C^{ x_1 x}.
\eeqq
Taking the additional derivative $X_{ y_2}$,  and then setting $ x_1= x_2= x$ gives 
$$2(\nabla^2\C( x, x)(X_ x,X_ x)-\nabla^2\C( x, x)(X_ x,X_ x))=0.
$$ 

Thus, since the variance here is zero, $ y= x$ is indeed a critical point of $ \tzeta( y, x)$, and so to evaluate 
the limit in \eqref{gettingthere:eq} we need yet  another round of  L'H\^opital's rule. This gives us that 
the limit is equal to
  \beqq
  \lim_{u\rightarrow 0}\frac{\frac{ {d}^2 \tzeta(c(u), x)}{ {d}u^2}}{-2\left[\left(\frac{ {d}\C(c(u), x)}{ {d}u}\right)^2+\C(c(u), x)\frac{ {d}^2\C(c(u), x)}{ {d}u^2}\right]}
\ = \frac{-\nabla^2 \tzeta( x, x)(X_ x,X_ x)}{2\nabla^2\C( x, x)(X_ x,X_ x)}
\eeqq
the equality here following from the fact that  $y= x$ is a critical point for both $ \tzeta( y,x)$ and $\C( y, x)$.

However, all terms here are well defined,  finite, and non-zero with probability one, so we are done.

\section{Proof of Lemma \ref{errorterm}}
\label{errproof}
To prove the lemma, we need to show that the sequence
\beqq
\sup_{x,y\in\tM} E^{x,k}(y)\ =
 \ \sup_{x,y\in\tM}\frac{k}{\|f^k(y)\|^2}\frac{(1-\C(x,y))^2}{(1-\widehat{\C}_k(x,y))^2}\left(\frac{1}{k}\|P_x f^{x,k}(y)\|^2\right)
\eeqq
converges to zero, with probability one.

By Lemmas \ref{firstterm} and \ref{secondterm} we know that each of the first two factors here a.s.\ converge, uniformly over $\tM$,  to one.
So it suffices to show the convergence of the final factor to zero, or that
\begin{equation}
\lim_{k\rightarrow\infty}\sup_{(x,y)\in\tM}\frac{1}{k}\|P_x f^{x,k}(y)\|^2\ =\ 0, \qquad a.s.
\label{needit}
\end{equation}
Since we have already shown that $f^x(y)$ is a.s.\ bounded over $\tM$, the absolute value of its supremum has (on a large deviations scale) Gaussian-like tails, and so standard Gaussian arguments show that the maximum of $k$ i.i.d.\ copies of this process can, asymptotically, be a.s.\ bounded by $C\sqrt{\log k}$ for some finite $C$. 

Since $P_x$ is orthogonal projection onto an $(m+1)$-dimensional subspace of $\mathbb{R}^k$, it now follows that, for large enough $k$,
\beq
\label{bound:sec10}
\frac{1}{k}\|P_x f^{x,k}(y)\|^2\ < \ \frac{C(m+1)\log k}{k},
\eeq
from which (\ref{needit}) now follows, and we are done.

\section{Fluctuation Theory for Local Reaches}
\label{sec:fluctuation}
We now return to the last part of Theorem \ref{maintheorem}, in which we described a fluctuation result involving the local reaches of the random manifolds $h^k(M)$. In particular, we want to consider the $k\to\infty$ distributional limit of the functions
\beq
\label{flucform:sec11}
\sqrt{k}\left(\cot^2\theta_k(\cdot)-\sigma^2_c(f,\cdot)\right)
\eeq
where $\theta_k(x)$, defined by \eqref{thetak:def}, is the local reach of $h^k(M)$ at the point $h^k(x)$, for $x\in M$.

The main result of this section is Theorem \ref{fluctheorem}, which contains what is needed to complete the statement of Theorem 
\ref{maintheorem}, in that the limit process for \eqref{flucform:sec11} is now described in  (formidable) detail.

To make that detail appear a little more natural, we shall do a little algebra before stating the theorem.

\subsection{Some algebra and rearrangements}

From the proof of Lemma \ref{lma:implemma}, we know that 
\beq
\cot^2 \theta_k(x) \
= \ 
\sup_{y\in M\setminus \{x\}} \left\{R_k(x,y)
\, -\, E^{x,k}(y)\right\},
\label{eq:locreach2}
\eeq
where $E^{x,k}(y)$ is the  `error' term defined at \eqref{errt} and we set
\beqq
R_k(x,y)\ \definedas\ 
\frac{k}{\|f^k(y)\|^2}\frac{(1-\C(x,y))^2}{(1-\widehat{\C}_k(x,y))^2}\left(\frac{1}{k}\|f^{x,k}(y)\|^2\right).
\eeqq
We already know from Lemma \ref{errorterm} that $E^{x,k}(y)\to 0$ uniformly in $x$ and $y$ as $k\rightarrow\infty$. However, looking back over the proof, in particular the final inequality \eqref{bound:sec10}, we see that the same is true for 
$\sqrt{k}E^{x,k}(y)$, from which it follows that we can ignore the error term in \eqref{eq:locreach2}. In addition, since we also know from Theorem \ref{maintheorem}  that 
\beq
\label{needed:eq}
\lim_{k\rightarrow\infty}\sup_{(x,y)\in\tM}\left|R_k(x,y)-\text{Var}(f^x(y))\right|=0.
\eeq
Thus it seems not unreasonable that the structure of the limit  of \eqref{flucform:sec11} might come from a continuous mapping theorem and the weak convergence of the random processes  $\gamma_k$, where
\begin{equation}
\gamma_k(x,y) \ \definedas \ \sqrt{k}\left(R_k(x,y)-\text{Var}(f^x(y))\right),\qquad (x,y)\in\tM.
\label{dlimit}
\end{equation}
If we now recall/introduce the  notations,
\beq
\label{Sigmak}
\Sigma^{(1)}_k(x) \ = \ \frac{\| f^k(x)\|^2}{k},\qquad \Sigma^{(2)}_k(x,y) \ =\  \frac{\|f^{x,k}(y)\|^2}{k},
\eeq
and
\beq
\label{Zk}
Z_k (x,y)  &=&  \sqrt{k}\left[\left(\frac{1-\C(x,y)}{1-\widehat{\C}_k(x,y)}\right)^2-1\right], \\
B_k(x,y)    &= & \sqrt{k}\left(  \Sigma^{(2)}_k(x,y)    -\text{Var}(f^x(y))\right),   \label{Bk}\\
N_k(y) &=& \sqrt{k}\left(\Sigma^{(1)}_k(y) - 1\right),
\label{Nk}
\eeq
%P_k (x,y)  &=&   \sqrt{k}\left(1- 1/\Sigma_k(x,y) \right)  \\
%   Q_k (x,y)  &=&   \Sigma_k(x,y) Z_k(x,y)
%\eeqq
%\beq
%\label{zedk}
%Z_k (x,y)   \ = \   \Sigma_k(x,y) \sqrt{k}\left(\left(\frac{1-\C^{xy}}{1-\widehat{\C}_k^{xy}}\right)^2-1\right)\
%\eeq
then it takes no more than a few lines of simple algebra to check that 
\beq
 \gamma_k(x,y)   % &\definedas &\sqrt{k}\left(R_k(x,y)-\text{Var}(f^x(y))\right) \\
\label{breakdown} &=&\ B_k(x,y) \,  -\,  N_k(y)Z_k(x,y)  \frac{\Sigma^{(2)}_k(x,y)}{\sqrt{k}\Sigma^{(1)}_k(y)} \\ 
\notag &&\qquad \qquad -\     \frac{\Sigma^{(2)}_k(x,y)}{\Sigma^{(1)}_k(y)}     N_k(y) \,  +\,  
\Sigma^{(2)}_k(x,y) Z_k(x,y)  .
\eeq
 
Now we wave our hands a little to come to some vague conclusions, before stating Theorem \ref{fluctheorem} which will make these conclusions precise, and then giving  proofs.
Firstly however, to reduce the lengths of some of the formulae to come, we recall some of our notational shorthand,
\beq
\label{shorthand}
\C^{xy} \ =\  \C(x,y),\quad
\wC_k^{xy} \ =\  \wC_k(x,y),
\eeq
and introduce
\beq
\label{varfxy-early} 
V^{xy}\ =\ \Var\left(f^x(y)\right)\ = \
\frac{1-(\C^{xy})^2-\sum_{i=1}^m (X_i\C^{xy})^2}{(1-\C^{xy})^2}.
\eeq
(To see why the right hand side here is indeed $\Var\left(f^x(y)\right)$, see
\eqref{varfxy} below.)

Now consider the various terms in \eqref{breakdown}.
 Although we did not state it explicitly, we have actually already proven that $N_k$ has a Gaussian limit. (See the discussion below in the proof of Lemma \ref{fidi-lemma-12}.) We also know, from Lemmas \ref{firstterm} and \ref{thirdterm},  that, as $k\to\infty$,   uniformly on $M$ and $\tM$,  respectively, 
\beqq
  \Sigma^{(1)}_k(x)\convas 1 \ \ \text{and}\ \  \Sigma^{(2)}_k(x,y)\convas V^{xy}.
 \eeqq
 In addition,  Lemma \ref{secondterm} and the a.s.\ convergence of $\Sigma^{(1)}_k$ lead to the expectation (this is the handwaving step) that $B_k$
 and $Z_k$ will both have Gaussian limits. 
Substituting this `information' into \eqref{breakdown}, the implication is that the first term on the right hand side will  have a Gaussian limit on $\tM$, the second will converge to zero, the third  will converge to $V^{xy}$ times a Gaussian process on $M$,  while the last term will converge to $V^{xy}$ times a Gaussian process on $\tM$. Unfortunately, all the limit processes will be correlated, which is what makes the precise description of this a little long-winded, as we now see.

%\begin{equation*}
%\sup_{y\in M\setminus \{x\}}\bigg[
%\sqrt{k}\left(\frac{\|f^{x,k}(y)\|^2}{k}-\text{Var}(f^x(y))\right)+\frac{\|f^{x,k}(y)\|^2}{k}\bigg\{\sqrt{k}\left(\frac{k}{\|f_k(y)\|^2}-1\right)
%\end{equation*}
%\begin{equation}
%\times\left(\left(\frac{1-\C^{xy}}{1-\widehat{\C}_k^{xy}}\right)^2-1\right)+\sqrt{k}\left(\frac{k}{\|f_k(y)\|^2}-1\right)+\sqrt{k}\left(\left(\frac{1-\C^{xy}}{1-\widehat{\C}_k^{xy}}\right)^2-1\right)\bigg\}\bigg].
%\label{eq:compo}
%\end{equation}
%We could claim that the limit is the supremum over $M\setminus \{x\}$ of the process which is the distributional limit of the sequence of processes in (\ref{eq:compo}), if we could prove fi-di convergence plus tightness of the sequence of processes formed by each of the four summands there \cite{van}.
%
%Recall from (\ref{eq:fxyrep}) that if we choose an orthonormal vector field $\{X_i(x)\}_{i=1}^m$ on $M$ with respect to the process induced $g$, we have that
%\begin{equation}
%f^x(y)\ =\ \frac{f(y)-\C(x,y)f(x)}{1-\C(x,y)}-\left(\sum_{i=1}^m {X}_i f(x)\,{X}_i \C(x,y)\right)\frac{1}{1-\C(x,y)}.
%\label{eq:representation}
%\end{equation}
%From now onwards, since Var$(f^x(y))$ appears very often, we denote it by $V^{xy}$. Let $C_b(\tM)$ denote the space of bounded, continuous functions over $\tM$. We show the weak convergence of the sequence $\{\sqrt{k}\left(R_k^{xy}-V^{xy}\right)\}$ to a Gaussian process $\gamma(x,y)$ in this space. This leads to the main result stated below. 

\subsection{The fluctuation result}

\begin{theorem}   
\label{fluctheorem}  
Let $f$ and $M$ satisfy the assumptions of Theorem \ref{maintheorem}, including the conditions that $M$ is $C^6$ and that, with probability one, $f\in C^6(M)$. Then
there exists a sequence $\bar\gamma_k$ of random processes from $M\to\real$, such that, for all $x\in M$,  
\beq
\label{bound-diff-11}
\sqrt{k}\left|\cot^2\theta_k(x)-\sigma^2_c(f,x)\right| \ \leq \ \left|\bar\gamma_k(x)\right|
\eeq
and a limit process $\bar\gamma\:M\to\real$ such that
\beq
\label{flucform:sec3-11}
\bar\gamma_k(\cdot)  \Rightarrow \ \bar\gamma (\cdot).
\eeq
The convergence here is weak convergence, in the Banach space $C_b(M)$ of  continuous functions on $M$ with supremum norm, and 
\beqq
\bar\gamma (x) \ = \ \sup_{y\in M\setminus x}\gamma(x,y), 
\eeqq
%where $\gamma$ is the Gaussian process over $\tM$ defined by
% \eqref{gammaeq}.
%
%
%
%$$
%\sqrt{k}(\cot^2\theta(h^k)(\cdot)-\sigma^2_c(f,\cdot))\ \to \bar\gamma (\cdot),
%$$ 
%The convergence  is weak convergence, in the Banach  space of bounded, continuous functions on $M$ with supremum norm, and 
%\beqq
%\bar\gamma (x) \ = \ \sup_{y\in M\setminus x}\gamma(x,y), 
%\eeqq
where $\gamma$ is the a.s.\ continuous Gaussian process over $\tM$ representable in distribution as
\begin{equation}
\gamma(x,y)\ = \ \beta(x,y)+V^{xy}\eta(y)+ 2V^{xy}\zeta(x,y)(1+\C(x,y)).
\label{gammaeq}
\end{equation}
Here
\begin{enumerate}
\item $\eta(y)$ is a centered Gaussian process over $M$ with correlation function 
$$\E\{\eta(y_1)\eta(y_2)\}=2(\C(y_1,y_2))^2.$$
\item $\beta(x,y)$ is a centered Gaussian process over $\tM$ with correlation function
$$\E\{\beta(x_1,y_1)\beta(x_2,y_2)\}=2\left(\E\{f^{x_1}(y_1)f^{x_2}(y_2)\}\right)^2.$$
\item $\zeta(x,y)$ is a centered Gaussian process over $\tM$ with correlation function
\begin{align*}
&\E \{ \zeta(x_1,y_1) \zeta(x_2,y_2)\}\\
&\quad =
 \frac{1}{(1-(\C^{x_1y_1})^{2})(1-(\C^{x_2y_2})^{2})}\\
&\qquad \times\bigg\{\smallhalf\C^{x_1 y_1}\C^{x_2y_2} \left[(\C^{y_1 x_2})^2+(\C^{y_1 y_2})^2+(\C^{x_1 x_2})^2+(\C^{x_1 y_2})^2\right]\qquad\\
&\qquad\qquad 
+\C^{x_2y_1}\left[\C^{x_1 y_2}-\C^{x_1 x_2}\C^{x_2y_2}\right]+\C^{y_1 y_2}\left[\C^{x_1 x_2}-\C^{x_1 y_2}\C^{x_2y_2}\right] 
\\
&\qquad\qquad\qquad\qquad\qquad -\C^{x_1 y_1}\left[\C^{x_1 x_2}\C^{x_1 y_2}+\C^{y_1 y_2}\C^{x_2y_1}\right]\bigg\}.
\end{align*}
\end{enumerate}
As for the corresponding  cross-covariance functions,  we write them in terms of   
 $X_1,\dots,X_m$,   an orthonormal (with respect to the induced metric)    vector field on $M$. None of the cross-covariances  are dependent on the particular choice of  vector field.
\begin{align*}
&\E\{\eta(y_1)\beta(x_2,y_2)\} =
\frac{ 2\left[\C^{y_1y_2}-\C^{x_2y_2}\C^{x_2y_1}-\sum_i X_i\C^{xy_1}|_{x=x_2} X_i\C^{xy_2}|_{x=x_2}
\right]^2}{(1-\C ^{x_2 y_2})^2},\\
%\eeqq
%\begin{align*}
&\E\{\eta(y_1)\zeta(x_2,y_2)\} =
\frac{2\C^{x_2y_1}\C^{y_1y_2}-\C^{x_2y_2}\left\{(\C^{x_2y_1})^2+(\C^{y_1y_2})^2\right\}}{1-(\C^{x_2 y_2})^2},
\end{align*}
%\\
%
\begin{align*}
&\E\{\zeta(x_1,y_1)\beta(x_2,y_2)\} \\
&=
\frac{1}{(1-(\C^{x_2 y_2})^2)}\\
&\qquad \times \left[\frac{2}{(1-\C^{x_1y_1})^2}\left\{\left(\C^{y_1y_2}-\C^{x_1y_1}\C^{x_1y_2}%\\
%&&\qquad\qquad\qquad\qquad 
-\sum X_i \C^{xy_1}|_{x=x_1} X_i \C^{xy_2}|_{x=x_1}\right)\right.\right.\\
&\qquad\qquad \qquad\qquad  \times\left.\left(\C^{y_1x_2}-\C^{x_1y_1}\C^{x_1x_2}-\sum X_i\C^{xy_1}|_{x=x_1}X_i\C^{xx_2}|_{x=x_1}\right)\right\}\\
&\qquad -\frac{\C^{x_2y_2}}{(1-\C^{x_1y_1})^2}\left\{\left(\C^{y_1y_2}-\C^{x_1y_1}\C^{x_1y_2}-\sum X_i\C^{xy_1}|_{x=x_1} X_i \C^{xy_2}|_{x=x_1}\right)^2\right.\\
&\qquad\qquad\qquad \  +\left.\left.\left(\C^{y_1x_2}-\C^{x_1y_1}\C^{x_1x_2}-\sum X_i \C^{xy_1}|_{x=x_1} X_i \C^{xx_2}|_{x=x_1}\right)^2\right\}\right].
\end{align*}
\label{thm:mainweak}
\end{theorem}

Although Theorem \ref{fluctheorem} takes a lot of space to state, its main implication is simple: The limiting fluctuations of the local reach of $h^k(M)$ are bounded by a functional of a Gaussian process on $\tM$. The detailed structure of this Gaussian process is complicated, and depends, in terms of properties such as differentiability, on the underlying covariance of $f$. For example, while the limit is a.s.\ continuous, it will not typically be differentiable, and fine sample path properties such as H\"older continuity will depend on the behaviour of the underlying covariance $\C$ in ways that are not at all obvious.

\section{Proof of  Theorem \ref{thm:mainweak}}
\label{mainweak:fidi}
We start with two lemmas, and then use these to complete the proof in the Section \ref{sec:mainweak:pf}.

\subsection{Two Lemmas}

\begin{lemma}
\label{fidi-lemma-12}
Under the conditions of Theorem \ref{fluctheorem}, and with the notation of the previous section,
we have the joint weak convergence of the following vector valued process over $\C_b(\tM)$:
\beq
\left(\Sigma^{(1)}_k,\,\Sigma^{(2)}_k,\,B_k,\,N_k,\,Z_k\right)\ \Rightarrow\ \ \left(1,\, V,\,\beta,\,\eta,\,2\zeta(1+\C)\right),
\label{fidis:eq:12} 
\eeq
where $V\:M\times M\to\real$ is defined by $V(x,y)=V^{xy}$.
%That is, we have the convergence of the finite dimensional (joint) distributions of all the random processes defined in 
%\eqref{Sigmak}     --  \eqref{Nk}. 
\end{lemma}
\begin{proof}

Most of the pieces that make up the proof of  Theorem \ref{thm:mainweak} are actually already in hand.
For a start,  by Lemmas \ref{firstterm} and \ref{thirdterm} we know $\Sigma^{(2)}_k(x,y)$ and $\Sigma^{(1)}_k(y)$ converge to the deterministic limits  $V^{xy}$ and 1, respectively, where the convergence is almost surely uniform in $(x,y)\in\tM$ and $y\in M$.  The corresponding weak convergence is, obviously, implied by this.  Secondly, in Section \ref{Sec:alpha} we established the weak convergence of $N_k$ in $C_b(M)$ (cf.\ \eqref{Nk:firsttime} and the last paragraph of Section 
\ref{Sec:932}).

To add $Z_k$ to this convergence, note that the main term there  -- $(1-\C(x,y))/(1-\widehat{\C}_k(x,y))$ -- already appeared in Section
\ref{sec:outhard}, and can be rewritten as  in  \eqref{long:eq}. Substituting  there the $\zeta_k(x,y)$ of \eqref{zetak:eq} and expanding  out the powers, simple algebra leads to the fact that
\beq
\label{zetak:yetagain}
Z_k(x,y) \ = \ 2(1+C(x,y))\zeta_k(x,y)\, + \, O(1/\sqrt{k}),
\eeq
and we have already shown the weak convergence of  $\zeta_k$ in $C_b(\tM)$ (cf.\ \eqref{zetak:again}). 

Note that to this point we have relied on results that arose in  earlier parts of the paper, and these required only that $f\in C^3(M)$. The additional level of differentiability required by the lemma, and so also by  Theorem \ref{fluctheorem}, comes from the following lemma, which  completes the collection by establishing the weak convergence of $B_k$.

In view of the fact that all the limit processes are either deterministic or Gaussian, applications of Slutsky's theorem and the 
Cram\'er-Wold device then complete the proof, modulo calculating all the the limit variances and covariances, for which we do not intend to write out the details.
\end{proof}

\begin{lemma}
Under the conditions of Lemma \ref{fidi-lemma-12}, and with the notation of the previous section,
$B_k\Rightarrow \beta$ in $C_b(\tM)$.
\end{lemma}

\begin{proof}
The proof follows along the same lines as the proof of the weak convergence of $\alpha_k$ described in Section \ref{Sec:alpha}.

To start, we once again choose an orthonormal frame field $\{X_i\}$ for $M$, with the conventions of Section \ref{sec:notation}. Write the corresponding Riemannian normal  basis vectors as $\{{\partial}/{\partial x^i}\}$,  and $\left\{{\partial}/{\partial x^i}:{\partial}/{\partial y^i}\right\}$ as the corresponding basis for the tangent spaces on $M\times M$. In this basis, we have
\begin{equation}
f^x(y) \ =\  \frac{f(y)-\C^{xy}f(x)}{1-\C^{xy}}-\sum_{i=1}^m\frac{ \frac{\partial f(x)}{\partial x^i} \,\frac{\partial \C^{xy}}{\partial x^i}}{1-\C^{xy}},
\label{eq:fxyrep-end}
\end{equation}
and 
\begin{equation}
V^{xy} \ = \text{Var}\left(f^x(y)\right) \  = \ \frac{1-(\C^{xy})^2-\sum_{i} (\frac{\partial \C^{xy}}{\partial x^i})^2}{(1-\C^{xy})^2}.
\label{varfxy}
\end{equation}
If we now write
\beqq
\Lambda_\ell(x,y) &\definedas&
(1-\C^{xy})^2 \,\left(\left(f_\ell^x(y)\right)^2-V^{xy}\right),
\eeqq
then we can also write 
\beq
\label{sum:lambdas}
B_k(x,y) \ = \ \frac{k^{-1/2}\sum_{\ell=1}^k   \Lambda_\ell(x,y)  }{(1-\C^{xy})^2}.
\eeq
 Suppose we can show that the numerator here has a Gaussian limit, $\Lambda$, say, as $k\to\infty$. Since it is a sum of i.i.d.\ processes, this should not be too hard. To complete the proof of the weak convergence of the $B_k$ we could then use 
 a continuous mapping argument, as before, by defining a map, $H$ say,   between functions on $\tM$ via
\beq
\label{eq:H}
(H(\phi)) (x,y) \ =\ \frac{\phi(x,y)}{(1-\C^{xy})^2},
\eeq
where the image function is in $C_b(\tM)$. For this to work, we need to know that  $H$ is continuous, with probability one, for the probability measure supported on the paths of $\Lambda$.  (This is not straightforward, since, as we shall soon see, we once again run into $0/0$ issues for $(H(\Lambda))(x,y)$ as $x\to y$.)   As a first step in  checking this continuity, we need to know something about $\Lambda$, and the function space on which the convergence of the numerator in \eqref{sum:lambdas} to $\Lambda$ occurs.

We start with $\Lambda$. Since, by assumption, it is mean zero  Gaussian, all of its properties are determined by  its covariance function. Given the expressions \eqref{eq:fxyrep-end} and \eqref{varfxy}, it is not hard to check that this is given by
\beq
&&\E\left\{\Lambda_\ell(x_1,y_1)\Lambda_\ell(x_2,y_2)\right\}  \nonumber \\
&&\quad =\E\left\{\Lambda(x_1,y_1)\Lambda(x_2,y_2)\right\}  \nonumber \\
&&\quad =\  \C^{y_1y_2}-\C^{x_2y_1}\C^{x_2y_2}  
%\nonumber \\ &&\qquad
-\ \sum_i \frac{\partial \C^{x_2y_1}}{\partial x^i} \frac{\partial \C^{x_2y_2}}{\partial x^i}+\C^{x_1x_2}\C^{x_1y_1}\C^{x_2y_2}\nonumber\\
&&\qquad +\ \C^{x_1y_1}\sum_i  \frac{\partial \C^{x_2x_1}}{\partial x^i} \frac{\partial \C^{x_2y_2}}{\partial x^i}-\sum_i  \frac{\partial \C^{x_1y_2}}{\partial x^i}\frac{\partial \C^{x_1y_1}}{\partial x^i}\nonumber\\
&&\qquad +\ \C^{x_2y_2}\sum_i\frac{\partial \C^{x_1x_2}}{\partial x^i}\frac{\partial \C^{x_1y_1}}{\partial x^i}+\sum_{i,j}\frac{\partial \C^{x_1y_1}}{\partial x^i}\frac{\partial \C^{x_2y_2}}{\partial x^i}\frac{\partial^2 \C^{x_1x_2}}{\partial x^i \partial x^j}.
%\label{eq:corrfxy}
\label{eq:corrL}
\eeq
(Note that setting $x=x_1=x_2$ and $y=y_1=y_2$ here is what gives the numerator in the expression for $V^{xy}$ in \eqref{varfxy-early}.)

We can now consider the behaviour of 
\beq
\lim_{y\rightarrow x}\frac{\Lambda(x,y)}{(1-\C^{xy})^2}.
\label{eq:limana}
\eeq
%Before starting on this, note that the correlation function $\E\{L(x_1,y_1)L(x_2,y_2)\}$ of $L(x,y)$ (the $y_i$ and $y^i$ should not be confused with each other) is straightforwardly seen to be
%and that $(1-\C^{x_1y_1})(1-\C^{x_2y_2})E\{f^{x_1}(y_1)f^{x_2}(y_2)\}$ equals

To see how this works, we restrict the argument to the case in which $M$ is one-dimensional. While notationally much simpler than the general case (although we shall see in a moment that it is hardly `simple') it is indicative of the general situation. In the general case the limit in \eqref{eq:limana} will be taken along a specific path of $y$'s, for which the final direction of approach to  $x$ will be what plays the role of the single dimension  in the following calculation.

Taking then $x,y\in M\subset \real^1$, it is an immediate consequence of \eqref{eq:corrL} that the variance of $\Lambda(x,y)$ tends to zero as $y\to x$, and thus so does $\Lambda(x,y)$ itself. The denominator here clearly also converges to zero, and so to compute the ratio we need to resort to an application of L'H\^opital's rule, which gives us that the limit in (\ref{eq:limana}) is the same as 
\beq
\lim_{y\to x} \frac{\frac{\partial \Lambda(x,x)}{\partial x}}{-2(1-\C^{xx})\frac{\partial \C^{xx}}{\partial x}}.
\eeq
Once again, it is obvious that the denominator  vanishes in the limit.

As for the numerator,  it follows from \eqref{eq:corrL}  and the fact that $g$-norm of $\frac{\partial}{\partial x}$ is one that 
\beqq
&&\E\left\{\left(\frac{\partial L(x,y)}{\partial x}\right)^2\right\}\bigg|_{y=x}\\
&&\qquad =\frac{\partial}{\partial y_1}\frac{\partial}{\partial y_2}\E\{L(x,y_1)L(x,y_2)\}|_{y_1=y_2=x}\\
&&\qquad = 4\left(\C^{y_1y_2}-\C^{xy_1}\C^{xy_2}-\frac{\partial \C^{xy_1}}{\partial x}\frac{\partial \C^{xy_2}}{\partial x}\right)\\
&& \qquad\qquad\qquad\qquad  \times\left(\frac{\partial^2 \C^{y_1y_2}}{\partial y_1 \partial y_2}-\frac{\partial \C^{xy_1}}{\partial y_1}\frac{\partial \C^{xy_2}}{\partial y_2}-\frac{\partial^2 \C^{xy_1}}{\partial y_1\partial x}\frac{\partial^2 \C^{xy_2}}{\partial y_2\partial x}\right)\\
&&\qquad\qquad +\ 4\left(\frac{\partial \C^{y_1y_2}}{\partial y_1}-\frac{\partial \C^{xy_1}}{\partial y_1}\C^{xy_2}-\frac{\partial^2 \C^{xy_1}}{\partial y_1\partial x}\frac{\partial \C^{xy_2}}{\partial x}\right)\\
&& \qquad\qquad\qquad\qquad  \times\left(\frac{\partial \C^{y_1y_2}}{\partial y_2}-\frac{\partial \C^{xy_2}}{\partial y_2}\C^{xy_1}-\frac{\partial^2 \C^{xy_2}}{\partial y_2\partial x}\frac{\partial \C^{xy_1}}{\partial x}\right).
\eeqq
However, evaluated at $y_1=y_2=x$, this also vanishes, so yet another application of 
 L'H\^opital's rule is required.
 
 In fact, two more applications of L'H\^opital's rule are required, and while the derivation  follows the line of the previous applications, the formulae are rather long, and so we will skip the details. However, in the end, one finds that  
\beq
\lim_{y\rightarrow x}\frac{\Lambda(x,y)}{(1-\C^{xy})^2}\ = \
 \frac{\frac{\partial^4 L(x,x)}{\partial x^4}}{6\left(\frac{\partial^2 \C^{xx}}{\partial x^2}\right)^2},
\label{eq:ratio:r}
\eeq
where the variance of the numerator is \beq
72\left(\frac{\partial^4 \C^{xx}}{\partial x^4}-\left(\frac{\partial^3 \C^{xx}}{\partial x^3}\right)^2-1\right)^2,
\eeq
which, like the denominator of \eqref{eq:ratio:r}  is non-zero. (This is a consequence of   the non-degeneracy assumed in
Assumption \ref{f:assump}.)

The punch line to all this is that in order to apply the continuous mapping theorem with the mapping $H$ of \eqref{eq:H}, we need to have convergence not only of the sum $k^{-1/2}  \sum\Lambda_\ell$, but also at least four of its derivatives. That is, we need weak convergence in the Banach  space $B^{(4)}$ of four times  continuously differentiable functions on $M\times M$, equipped with the norm 
\beqq
\qquad \ 
\|f\|_{B^{(4)}}  \definedas   \max\left\{ \|f\|_{\infty},\, \|\nabla f\|_{\infty},\, \|\nabla^2 f\|_{\infty}, \|\nabla^3 f\|_{\infty},\, \|\nabla^4 f\|_{\infty}\right\},
%\notag
\eeqq
(cf.\ \eqref{Bnorm:defn}).  %\\

Now that we know what to do, the rest is, at least in principle, straightforward, and 
 the proof follows along the same lines of the proof of the weak convergence of $\alpha_k$ we treated in Section \ref{Sec:alpha}. Convergence of finite dimensional distributions follows from Theorem \ref{Anderson:thm}, while tightness requires the computation of moments of increments of the $\Lambda_\ell$ and their first four derivatives. Note, however,  that $\Lambda_\ell(x,y)$ involves  $f_\ell^x(y)$. Since we have seen that $f_\ell^x(\cdot)$, as a function on $M$,   basically possesses one less level of differentiability that $f$ itself, requiring four derivatives for  $\Lambda_\ell$ ultimately  leads to requiring $f\in C^5(M)$. In addition, since  the arguments is Section  \ref{Sec:alpha} relied on a Taylor expansion,  one further derivative is required, which is why the lemma, and so Theorem \ref{fluctheorem}, require $f\in C^6(M)$.
 
 We leave the details to the reader. While they are long and involved, the fact that all random variables are either Gaussian or squares of Gaussians means that there is no more involved than Wick's formula and accounting.
 \end{proof}

\subsection{Proof of  Theorem \ref{thm:mainweak}}
\label{sec:mainweak:pf}  %{mainweak:final}

From \eqref{eq:locreach2}, \eqref{dlimit} and the definition \eqref{sigmax:def} of  $\sigma_c^2(f,x))$ we have that 
\beqq
&&\sqrt{k}\left|\cot^2\theta_k(x)-\sigma_c^2(f,x)\right|\\
&&\qquad\quad =\  \sqrt{k}\,\Big|\sup_{y\in M\setminus\{x\}}\left(R^{xy}_k  -E^{x,k}(y)\right) -\sup_{y\in M\setminus\{x\}}V^{xy}\Big|\\
&&\qquad \quad \leq\ \sqrt{k}\,\Big|\sup_{y\in M\setminus\{x\}}R^{xy}_k  -\sup_{y\in M\setminus\{x\}}V^{xy}\Big|
\ +\ \sqrt{k}\sup_{y\in M\setminus\{x\}}\ \big| E^{x,k}(y)\big|\\
&&\qquad\quad \leq \ \sup_{y\in M\setminus\{x\}}\left|\gamma_k(x,y)\right| \ +\ \sqrt{k}\sup_{y\in M\setminus\{x\}}\ \big|E^{x,k}(y)\big|.
\eeqq
From the discussion preceding \eqref{needed:eq} we know that we can ignore the second term here in the limit. The 
representation of $\gamma_k$ in \eqref{breakdown} in terms of the processes $\Sigma_k^{(1)}$, $\Sigma_k^{(2)}$ $B_k$, $N_k$ and $Z_k$, the joint weak convergence of all of these in Lemma \ref{fidi-lemma-12}, and an application of the continuous mapping theorem, complete the proof of Theorem \ref{thm:mainweak}.
\qed

%The lemmas proved in this section give that the weak limit of $\gamma_k$ equals the Gaussian limit process 
%\begin{equation*}
%\gamma(x,y)\definedas \beta(x,y)+V^{xy}\eta(y)+ 2V^{xy}\zeta(x,y)(1+\C(x,y)).
%\end{equation*}
%Since
%$$\hspace{-3.4in}\sqrt{k}(\cot^2\theta_k(x)-\sigma_c^2(f,x))$$
%\begin{eqnarray*}
%&=&\sqrt{k}\bigg(\sup_{y\in M\setminus\{x\}}R^{xy}_k-\sup_{y\in M\setminus\{x\}}V^{xy}\bigg)+O(\log k/\sqrt{k})\\
%&\leq& \sup_{y\in M\setminus\{x\}}\gamma_k(x,y)+O(\log k/\sqrt{k}),
%\end{eqnarray*}
%the proof of Theorem \ref{thm:mainweak} is done. Since we proved weak convergence in the space of processes with a.s. bounded and continuous sample paths, the limit process
%$\gamma(x,y)$
%also has a.s. bounded, continuous sample paths. It might be interesting to look at further questions about the regularity of these sample paths and the exact growth rate of the fluctuations.

\section*{Appendix 1}
We now give a  proof of Lemma \ref{crm}. As mentioned earlier,
Lemma \ref{crm} is identical to Lemma 2.1  of \cite{Takemura-equivalence} and, as pointed out there, the proof is essentially the same as the proof given in \cite{Johansen} for the one-dimensional case. Thus, we make no claim of originality, and include the proof for completeness only. 

\begin{proof}
Take $\eta_x\in T^{\perp}_xM\cap S(T_x S^{k-1})$, and consider  the geodesic 
\beqq
\gamma_{(x,\eta_x)}(r)=\cos r  \,x+\sin r \, \eta_x,\qquad r\geq 0.
\eeqq
To determine the local reach in the direction $\eta_x$, we need to know how far we can extend $\gamma$ so that the metric projection of the endpoint is $x$. Clearly, this is until we find a $y\neq x$ such that 
\beqq
\langle y,\gamma_{(x,\eta_x)}(r)\rangle\ =\ \langle x,\gamma_{(x,\eta_x)}(r)\rangle.
\eeqq
Consider first the case of $r<\frac{\pi}{2}$ so that $\cos r>0$.  Then the above two formulae imply that we can  extend the geodesic  at least until such an $r$ as long as 
\beqq
&&\sup_{y\neq x}\,(\cos r  \,(\langle y,x\rangle-1)+\sin r  \langle y,\eta_x\rangle)\leq 0 \\
&&\qquad \iff \sup_{y\neq x} \,(-\cos r  \,(1-\langle y,x\rangle -1)+\sin r  \, \langle y,\eta_x\rangle)\leq 0 \\
&&\qquad \iff \sup_{y\neq x}\left(-\cot r +\frac{\langle y,\eta_x\rangle}{1-\langle x,y\rangle}\right)\leq 0\\
&&\qquad \iff\sup_{y\neq x}\frac{\langle y,\eta_x \rangle}{1-\langle x,y\rangle}\leq \cot r.
\eeqq
When $r\geq \frac{\pi}{2}$, the same argument  gives that as soon as there is a $y$ such that $\langle y,\eta_x\rangle >0$, 
\beqq
\cos r\,(1-\langle y,x\rangle)+\sin r\,\langle y,\eta_x\rangle\ >\ 0,
\eeqq
and therefore, for such $r$,
\beqq
\sup_{y\neq x}\,(\cos r  \,(\langle y,x\rangle-1)+\sin r  \langle y,\eta_x\rangle)\ > \ 0,
\eeqq
implying that such a $y$ is closer to $\gamma_{x,\eta_x}(r)$ than $x$ is. Hence, the geodesic can only be extended up to a length less than or equal to $\frac{\pi}{2}$.\\
\indent Thus, by our earlier argument for $r\leq \frac{\pi}{2}$, we note that on the set
$$Z=\big\{(x,\eta_x)\: \sup_{y\neq x}\langle y,\eta_x\rangle>0\big\},$$
the critical radius satisfies
$$\cot r(x,\eta_x)\ =\ \sup_{y\neq x}\frac{\langle y,\eta_x \rangle}{1-\langle x,y\rangle}\ =\ \sup_{y\neq x}\frac{\langle y,\eta_x \rangle^+}{1-\langle x,y\rangle}\ \geq \ 0,$$
where $x^{+}$ denotes the positive part of $x$. Meanwhile, on the set $Z^{\text{c}}$, we have $$\cot r(x,\eta_x)\leq 0.$$ Therefore, we have the inequality,
$$\cot r(x,\eta_x)\ \leq \ \sup_{y\neq x}\frac{\langle y,\eta_x \rangle^+}{1-\langle x,y\rangle},$$
which becomes an equality when
$$\sup_{y\neq x}\frac{\langle y,\eta_x \rangle^+}{1-\langle x,y\rangle}\ >\ 0.$$
In other words, we have
$$
\cot\left( \min\left(r(x,\eta_x),\frac{\pi}{2}\right)\right)\ =\  \sup_{y\neq x}\frac{\langle y,\eta_x \rangle^+}{1-\langle x,y\rangle}.
$$
Finally, since $M$ is a closed manifold embedded into a sphere, the local reach at $x$, which is an infimum over all $\eta_x$ above, cannot be greater than $\frac{\pi}{2}$. Thus we can truncate at this angle, and so,  
% that there is no problem with truncating the local critical radius of the pair $(x,\eta_x)$ to  on the set $Z^{\text{c}}$. That is, we take the local reach, or critical radius for the pair $(x,\eta_x)$ to be
%\beqq
%\theta_\ell(x,\eta_x)\ = \ \min\left(r(x,\eta_x),{\pi}/{2}\right),
%\eeqq
%where $r(x,\eta_x)$ is computed as in (\ref{locrad}), and the local reach at $x$ is given 
by (\ref{locx}), \eqref{globc}, and the above, obtain that 
%. This does not affect the values of the local reach at $x$ (and therefore, the global reach of $M$) since to each $(x,\eta_x)$ in $Z^{\text{c}}$, we can associate a different pair $(x,\eta'_x)$ which would fall within $Z$.\\
%We then have
%\beqq
%\theta_\ell (x,\eta_x)\ =\ \cot^{-1}\left(\sup_{y\neq x}\frac{\langle y,\eta_x \rangle^+}{1-\langle x,y\rangle}\right).$$
%The critical radius at $x$ is then defined by the relation
\beqq
\cot^2(\theta(x))\ =\ \cot^2\left(\inf_{\eta_x}\theta_\ell(x,\eta_x)\right) &=& \sup_{\eta_x:\|\eta_x\|=1}\sup_{y\neq x}\left(\frac{\langle y,\eta_x \rangle^+}{1-\langle x,y\rangle}\right)^2 \\
&=&\sup_{y\neq x}\frac{\|P_{T^{\perp}_x M}y\|^2}{(1-\langle x,y\rangle)^2},
\eeqq
as required.
\end{proof}

\section*{Appendix 2}
%\label{app:estimate}
We need to show that the remainder terms, $O(1/k^2)$, in  (\ref{bias:eq}) and (\ref{variance:eq}) are of the right order and, just as importantly, are uniform over $M\times M$.

As mentioned earlier, if the correlation estimates $\widehat \C (x,y)$ of $\C(x,y)$ were centered at sample means rather than zero -- which  we shall refer to as the `standard' case --  we could simply quote known results from the Statistics literature to establish everything we need. These results are not hard to prove, but they involve pages of tedious algebra, which we do not want to try to reproduce here. Rather, we shall suffice with describing the standard proofs, and where changes need to be made to cover our situation.

The standard  case is treated in \cite{Kendall}.  Following the derivation in Chapter 16, Section 16.24 there, we start by writing out the joint probability of $k$ sample values $\{(f_j(x),f_j(y))\}_{j=1}^k$ drawn from a bivariate normal density with zero means and unit variances in terms of the statistics we are interested in, namely, 
\beqq
s_1^2\ \definedas \ \frac{1}{k}\sum f_j^2(x), \qquad s_2^2\definedas \frac{1}{k} \sum f_j^2(y),
\eeqq 
along with the $\wC_k(x,y)$ of (\ref{MLest}). These replace the standard  sample mean centered version of these statistics in   \cite{Kendall}.

 Using the result of Example 11.6 in Chapter 11 of \cite{Kendall} which deals with finding the distribution of a sum of squares of i.i.d.\ standard normal variates, and following the discussion in Section 16.24 there, we find 
 %that the differential element of the distribution is proportional to $s_1^{k-1}s_2^{k-1}ds_1ds_2(1-\wC^2)^{\frac{k-3}{2}}d\wC.$ The only difference between this and the corresponding quantity for the statistics centered around sample means is that $k$ is changed to $k+1$ (see (16.51) in \cite{Kendall}).\\
that the exact joint probability density of $s_1,s_2,$ and $\wC_k(x,y)$, on $\real_+\times\real_+\times [-1,1]$, is given by 
\beqq
&&\frac{k^k s_1^{k-1}s_2^{k-1}(1-\wC_k^2(x,y))^{\frac{k-3}{2}}}{\pi\Gamma(k-1)(1-\C^2(x,y))^{k/2}}\times e^{-\frac{k}{2(1-\C^2(x,y))}\left(s_1^2-2\C(x,y)\wC_k(x,y) s_1s_2+s_2^2\right)}.
\label{eq:density}
\eeqq
As in Section 16.32 of \cite{Kendall}, we now integrate out $s_1$ and $s_2$, and use the remaining density  of $\wC$ to compute that \beqq
\E\{\wC_k(x,y)\}
\ = \ \frac{\C(x,y)\,\Gamma^2((k+1)/2)}{\Gamma(k/2)\,\Gamma((k+2)/2)}\,F\left(\frac{1}{2},\frac{1}{2},\frac{k+2}{2},\C^2(x,y)\right),
\label{eq:meanC}
\eeqq
where $F$ is the hypergeometric function. 
Note that $0\leq\C^2(x,y)\leq 1$ for all $(x,y)\in M\times M$. The fact that $$F(\alpha,\beta,\gamma,x)=1+xO(1/\gamma)\,\,\text{as}\,\,\gamma\rightarrow\infty$$ uniformly for $x$ in any bounded set (cf. \cite{Olkin}), and a Stirling's approximation which gives that the ratio of Gamma functions in (\ref{eq:meanC}) converges to $1$ as $k\rightarrow\infty$, gives the uniformity of the error bound in (\ref{bias:eq}). A similar calculation establishes (\ref{variance:eq}) and the uniformity of the error bound there.

\section*{Acknowledgements}
We would like to thank Takashi Owada for useful discussions.

\bibliographystyle{plain}

\bibliography{tau-bib}

\end{document}